\documentclass[journal]{IEEEtran}

\usepackage[cmex10]{amsmath}
\usepackage{amsthm,aliascnt,hyperref}
\usepackage{amssymb}
\usepackage{balance}
\usepackage{url}
\usepackage{pdfsync}
\usepackage{xr}

\newcommand{\icl}[1]{{\color{black}{#1}}}

\makeatletter

\newcommand\Autoref[1]{\@first@ref#1,@}
\def\@throw@dot#1.#2@{#1}
\def\@set@refname#1{
	\edef\@tmp{\getrefbykeydefault{#1}{anchor}{}}%
	\def\@refname{\@nameuse{\expandafter\@throw@dot\@tmp.@autorefname}s}%
}
\def\@first@ref#1,#2{%
	\ifx#2@\autoref{#1}\let\@nextref\@gobble
	\else%
	\@set@refname{#1}
	\@refname~\ref{#1}
	\let\@nextref\@next@ref
	\fi%
	\@nextref#2%
}
\def\@next@ref#1,#2{%
	\ifx#2@ and~\ref{#1}\let\@nextref\@gobble
	\else, \ref{#1}
	\fi%
	\@nextref#2%
}

\makeatother

\newcommand{\mynewtheorem}[2]{
  \newaliascnt{#1}{dummy}
  \newtheorem{#1}[#1]{#2}
  \aliascntresetthe{#1}
  \expandafter\def\csname #1autorefname\endcsname{#2}
}

\theoremstyle{plain}
  \mynewtheorem{theorem}{Theorem}
  \mynewtheorem{proposition}{Proposition}
  \mynewtheorem{corollary}{Corollary}
  \mynewtheorem{lemma}{Lemma}
  \mynewtheorem{primal}{Primal}
  \mynewtheorem{dual}{Dual}
  \mynewtheorem{example}{Example}

\theoremstyle{definition}
  \mynewtheorem{definition}{Definition}
  \mynewtheorem{problem}{Problem}
\theoremstyle{remark}
  \mynewtheorem{remark}{Remark}

\newcommand{\appendixref}[1]{\hyperref[#1]{Appendix~\ref{#1}}}

\usepackage{amsfonts}
\usepackage{array}
\usepackage{verbatim}
\usepackage{enumerate}
\usepackage{tikz}
\usetikzlibrary{arrows,shapes,calc,through,backgrounds,intersections,positioning}

\tikzset{
    right angle quadrant/.code={
        \pgfmathsetmacro\quadranta{{1,1,-1,-1}[#1-1]}     
        \pgfmathsetmacro\quadrantb{{1,-1,-1,1}[#1-1]}},
    right angle quadrant=1, 
    right angle length/.code={\def\rightanglelength{#1}},   
    right angle length=2ex, 
    right angle symbol/.style n args={3}{
        insert path={
            let \p0 = ($(#1)!(#3)!(#2)$) in     
                let \p1 = ($(\p0)!\quadranta*\rightanglelength!(#3)$), 
                \p2 = ($(\p0)!\quadrantb*\rightanglelength!(#2)$) in 
                let \p3 = ($(\p1)+(\p2)-(\p0)$) in  
            (\p1) -- (\p3) -- (\p2)
        }
    }
}
\usepackage{color}

\DeclareMathOperator{\diag}{diag}

\DeclareMathOperator{\Span}{span}
 \DeclareMathOperator*{\interior}{int}
  \DeclareMathOperator*{\relint}{relint}
  \DeclareMathOperator*{\Conv}{Conv}
  \DeclareMathOperator*{\affinespan}{aff}

\usepackage{graphicx}
\usepackage{epstopdf}

\begin{document}
\title{Stability and instability \\ in saddle point dynamics - Part I}
\author{Thomas Holding\thanks{\color{black}Thomas Holding is with the Department of Mathematics, Imperial College London, United Kingdom; \tt\small t.holding@imperial.ac.uk} and Ioannis~Lestas\thanks{Ioannis Lestas is with the Department of Engineering, University of Cambridge, Cambridge, CB2 1PZ, United Kingdom;
\tt\small icl20@cam.ac.uk}\thanks{\color{black}Paper \cite{holding2014convergence} is a conference paper that includes preliminary results in this manuscript (Part I). This manuscript includes additional results of independent interest, detailed derivations and extensive discussion.}
}

\markboth{}%
{}

\maketitle
\begin{abstract}\color{black}
We consider the problem of convergence to a saddle point of a concave-convex function via gradient dynamics.
Since first introduced by Arrow, Hurwicz and Uzawa in \cite{arrow} such dynamics have been extensively used in diverse areas, there are, however, features that render their analysis non trivial. These include the lack of convergence guarantees when the function considered is not strictly concave-convex and also the non-smoothness of subgradient dynamics.
Our aim in this two part paper is to provide an explicit characterization to the asymptotic behaviour of general gradient and subgradient dynamics applied to a general concave-convex function. We show that
despite the nonlinearity and non-smoothness of these dynamics their $\omega$-limit set is comprised of trajectories that solve only explicit {\em linear} ODEs that are characterized within the paper.

More precisely, in Part I an exact characterization is provided to the asymptotic behaviour of unconstrained gradient dynamics. 
We also show that when convergence to a saddle point is not guaranteed then the system behaviour can be problematic, with arbitrarily small noise leading to an unbounded variance.
In Part II we consider a general class of subgradient dynamics that restrict trajectories in an arbitrary convex domain, and show that \icl{when an equilibrium point exists} the limiting trajectories are solutions of subgradient dynamics on only affine subspaces. The latter is a smooth class of dynamics with an asymptotic behaviour exactly characterized in Part I, as solutions to explicit linear ODEs. These results are used to formulate corresponding convergence criteria and are demonstrated with several examples and applications presented in Part II.
%
%
\end{abstract}



\IEEEpeerreviewmaketitle

%
%
%

\section{Introduction}
\IEEEPARstart{F}{inding} the saddle point of a concave-convex function is a problem that is relevant in many applications in engineering and economics and has been addressed by various communities. It includes, for example, optimization problems that are reduced to finding the saddle point of a Lagrangian.
The gradient method, 
first introduced by Arrow, Hurwicz and Uzawa \cite{arrow} has been widely used in this context as it leads to decentralized update rules for network optimization problems. It has therefore been extensively used in areas such as resource allocation in communication and economic networks (e.g. \cite{Hurwicz}, \cite{KMT}, \cite{SrikantB}, \cite{Paganini}, \cite{low1999optimization}, \cite{Lestas-Routing-CDC2004}), {\color{black}game theory \cite{gharesifard2013distributed}, distributed optimization \cite{gharesifard2014distributed, wang2011control, richert2015robust} and power networks \cite{Power1, Power2, Power3, Power3b, devane2016distributed, stegink2017unifying, li2016connecting, mallada2017optimal}}.




Nevertheless,
in broad classes of problems there are features that render the analysis of the asymptotic behaviour of gradient dynamics nontrivial.
In particular,
{\color{black}even though}
for a strictly concave-convex function convergence to a saddle point via gradient dynamics is ensured, {\color{black}when this strictness is lacking,} convergence is not guaranteed and oscillatory solutions can occur. The existence of such oscillations has been reported in {\color{black}various} applications \cite{arrow}, \cite{Paganini}, \cite{Holding-Lestas-CDC2013}, \cite{Goran}, however, an exact characterization {\color{black}of their explicit form for a general concave-convex function, which leads also to a necessary and sufficient condition for their existence, has not been provided} in the literature and is one of the aims of Part~I of this work.   

Furthermore, when subgradient methods are used to restrict the dynamics in a convex domain (needed, e.g., in optimization problems),  the dynamics become non-smooth in continuous-time. This increases significantly the complexity in the analysis as classical Lyapunov and LaSalle type techniques (e.g. \cite{khalil}) cannot be applied. This is also reflected in the alternative approach taken for the convergence proof in \cite{arrow} for subgradient dynamics applied to a strictly concave-convex Lagrangian with positivity constraints. \icl{From an early stage it has been noted in the literature that the right-hand side of (sub)gradient dynamics is monotone \cite{rockafellar1971saddle}.
This has been exploited to derive convergence results under appropriate strictness in the concavity/convexity properties \cite{venets1985continuous}, \cite{goebel2017stability}.}
\icl{In a more recent study \cite{Cherukuri-primal-dual} it was} pointed out that the invariance principle for hybrid automata in \cite{Hybrid} cannot be applied in this context, and gave an alternative proof, by means of Caratheodory's invariance principle, to the convergence result in \cite{arrow} mentioned above. {\color{black} Convergence criteria for unconstrained gradient dynamics
were also derived in \cite{cherukuri2017saddle} and under positivity constraints in \cite{cherukuri2016role}. In general, rigorously proving convergence for the subgradient method, even in what would naively appear to be simple cases, is a non-trivial problem, and requires much machinery from non-smooth analysis \cite{corts2008discontinuous, goebel2009hybrid}.}


Our aim in this two part paper is to provide an explicit characterization 
of the asymptotic behaviour of continuous-time gradient {\color{black}and subgradient dynamics applied to a general concave-convex function.
 Our analysis is carried out in a general setting, where the function with respect to which these dynamics  are applied is not necessarily {strictly} concave-convex.   {\color{black}Furthermore, a general class of subgradient dynamics are considered, where trajectories are restricted in an arbitrary convex domain}. 
 One of our main results is to show that despite the nonlinear and nonsmooth character of these dynamics their $\omega$-limit set is comprised of trajectories that solve explicit {\em linear}~ODEs.}

Our main contributions can be summarized as follows:
\begin{itemize}
\item In {\color{black}Part I}, we consider the gradient method applied on a general concave-convex function in an unconstrained domain, and provide an exact characterization to the limiting solutions, which can in general be oscillatory. In particular, we show that despite the nonlinearity of the dynamics the trajectories converge to solutions that satisfy a linear ODE that is explicitly characterized. Furthermore, we show that when such oscillations occur, the dynamic behaviour can be problematic, in the sense that arbitrarily small stochastic perturbations lead to an unbounded variance, \icl{when the set of saddle points includes a bi-infinite line}.
    \item In {\color{black}Part II}, we consider the subgradient method applied to a general concave-convex function with the trajectories restricted in an arbitrary convex domain, \icl{such that an equilibrium point exists}. We show that despite the non-smooth character of these dynamics,  {\color{black} their limiting behaviour is given by the solutions of one of an explicit family of {\color{black}linear} {\color{black}ODEs. In particular, these ODEs are shown to be solutions of subgradient dynamics on affine subspaces, which is a class of dynamics the asymptotic properties of which are exactly determined in Part I.
        These results are used to formulate corresponding convergence criteria,
        and various examples and applications are discussed.}}
\end{itemize}

It should be noted that
there is a direct link between the results in Part I and Part II as the dynamics, that are proved to be associated with the asymptotic behaviour of the subgradient method, are a class of dynamics that can be analysed with the framework introduced in Part I. Applications of the results in Part I will therefore be discussed in Part II, as in many cases (e.g. optimization problems with inequality constraints) a restricted domain for the concave-convex function needs to be considered.

Finally, we would also like to comment that the methodology used for the derivations in the two papers is of independent technical interest. In Part I the analysis is based on various geometric properties established for the saddle points of a concave-convex function. In Part II the non-smooth analysis is carried out by means of some more abstract results on
{\color{black}corresponding semiflows}
that are applicable in this context, while also making use of the notion of a face of a convex set to characterize the asymptotic behaviour of the dynamics.
%


The Part I paper is structured as follows. In section \ref{sec:Preliminaries} we introduce various definitions and preliminaries that will be used throughout the paper. In section \ref{sec:problem_formulation} the problem formulation is given and the main results are presented in section \ref{sec:Main}, i.e. characterization of the limiting behaviour of gradient dynamics. This section also includes an extension to a class of subgradient dynamics that restrict the trajectories on affine spaces. This is a technical result that will be used in Part~II to characterize the limiting behaviour of general subgradient dynamics. The proofs of the results are finally given in section~\ref{sec:proofs}.

\section{Preliminaries}
\label{sec:Preliminaries}
\subsection{Notation}
Real numbers are denoted by $\mathbb{R}$ and non-negative real numbers as $\mathbb{R}_+$. For vectors $x,y\in\mathbb{R}^n$ the inequality $x<y$ denotes the corresponding element wise inequality, i.e. $x_i<y_i \ \forall i$,
$d(x,y)$ denotes the Euclidean metric and $|x|$ denotes the Euclidean norm.


The space of $k-$times continuously differentiable functions is denoted by $C^k$. For a sufficiently differentiable function $f(x,y):\mathbb{R}^n\times\mathbb{R}^m\to\mathbb{R}$ we denote the vector of partial derivatives of $f$ with respect to $x$ as $f_x$, respectively $f_y$. The Hessian matrices with respect to $x$ and $y$ are denoted $f_{xx}$ and $f_{yy}$,
{\color{black}while $f_{xy}$ denotes the matrix of partial derivatives defined as $[{f_{xy}}]_{ij}:=\frac{\partial f}{\partial x_i \partial y_j}$.}
For a vector valued function $g:\mathbb{R}^n\to\mathbb{R}^m$ we let $g_x$ denote the matrix formed by partial derivatives of the elements of $g$, {\color{black}i.e. $[g_x]_{ij}=\frac{\partial g_i}{\partial x_j}$}.

For a matrix $A\in\mathbb{R}^{n\times m}$ we denote its kernel and transpose by $\ker(A)$ and $A^T$ respectively. If $A$ is in addition symmetric, we write $A<0$ if $A$ is negative definite.

\subsubsection{Geometry}\label{sec:geometry}

For subspaces $E\subseteq\mathbb{R}^n$ we denote the orthogonal complement as $E^\perp$, and for a set of vectors $E\subseteq\mathbb{R}^n$ we denote their span as $\Span(E)$, their affine hull as $\affinespan(E)$ and their convex hull as $\Conv(E)$. \icl{A set $K\subset\mathbb{R}^n$ is a  bi-infinite line if it is the affine hull of two distinct points in $\mathbb{R}^n$.} The addition of a vector $v\in\mathbb{R}^n$ and a set $E\subseteq\mathbb{R}^n$ is defined as $v+E=\{v+u:u\in E\}$.

For a set $K\subset\mathbb{R}^n$, we denote the interior, relative interior, boundary and closure of $K$ as $\interior K$, $\relint K$, $\partial K$ and $\overline{K}$ respectively, and we say that $K$ and $M$ are orthogonal and write $K\perp M$ if for any two pairs of points $\mathbf{k},\mathbf{k'}\in K$ and $\mathbf{m},\mathbf{m'}\in M$, we have $(\mathbf{k}'-\mathbf{k})^T(\mathbf{m}-\mathbf{m'})=0$.

Given a set $E\subseteq\mathbb{R}^n$ and a function $\phi:E\to E$ we say that $\phi$ is an isometry of $(E,d)$ or simply an isometry, if for all $x,y\in E$ we have $d(\phi(x),\phi(y))=d(x,y)$.

For $x\in\mathbb{R},y\in\mathbb{R}_+$ we define $\left[x\right]_y^+=x$ if $y>0$ and $\max(0,x)$ if $y=0$.
\subsubsection{Convex geometry}\label{sec:conv_geom}

{\color{black}For a} closed convex set $K\subseteq\mathbb{R}^n$ and $\mathbf{z}\in\mathbb{R}^n$, we define the maximal orthogonal linear manifold to $K$ through $\mathbf{z}$ as
\begin{equation}\label{eq:orth_lin_man}
M_K(\mathbf{z})=\mathbf{z}+\Span(\{\mathbf{u}-\mathbf{u}':\mathbf{u},\mathbf{u}'\in K\})^\perp
\end{equation}
and the normal cone to $K$ through $\mathbf{z}$ as
\begin{equation}
\!\!\!N_K(\mathbf{z})=\{\mathbf{w}\in\mathbb{R}^n:\mathbf{w}^T(\mathbf{z}'-\mathbf{z})\le0\text{ for all }\mathbf{z}'\in K\}.
\end{equation}
When $K$ is an affine space $N_K(\mathbf{z})$ is independent of $\mathbf{z}\in K$ and is denoted $N_K$. If $K$ is in addition non-empty, then we define the projection of $\mathbf{z}$ onto $K$ as $\mathbf{P}_{K}(\mathbf{z})=\operatorname{argmin}_{\mathbf{w}\in K}d(\mathbf{z},\mathbf{w})$. 

\subsection{\color{black}Concave-convex functions and saddle points}
\begin{definition}[Concave-convex function]\label{def:concave-convex-function}
Let $K\subseteq\mathbb{R}^{n+m}$ be non-empty closed and convex. We say that a function $\varphi(x,y):K\to \mathbb{R}$ is \text{concave-convex on $K$} if for any $(x',y')\in K$, $\varphi(x,y')$ is a concave function of $x$ and $\varphi(x',y)$ is a convex function of $y$. If either the concavity or convexity is always strict, we say that $\varphi$ is \textit{strictly concave-convex on $K$}.
\end{definition}
\begin{definition}[Saddle point]\label{def:saddle-point}
For a concave-convex function $\varphi:\mathbb{R}^n\times\mathbb{R}^m\to\mathbb{R}$ we say that $(\bar{x},\bar{y})\in \mathbb{R}^{n+m}$ is a \textit{saddle point} of $\varphi$ if for all $x\in\mathbb{R}^n$ and $y\in\mathbb{R}^m$ we have the inequality $\varphi(x,\bar{y})\le\varphi(\bar{x},\bar{y})\le\varphi(\bar{x},y)$.
\end{definition}
If $\varphi$ is in addition $C^1$ then $(\bar{x},\bar{y})$ is a saddle point if and only if $\varphi_x(\bar{x},\bar{y})=0$ and $\varphi_y(\bar{x},\bar{y})=0$.

{\color{black}
When we consider a concave-convex function $\varphi(x,y):\mathbb{R}^n\times\mathbb{R}^m\to\mathbb{R}$ we shall denote the pair $\mathbf{z}=(x,y)\in\mathbb{R}^{n+m}$ in bold, and write $\varphi(\mathbf{z})=\varphi(x,y)$. The full Hessian matrix will then be denoted $\varphi_{\mathbf{zz}}$. Vectors in $\mathbb{R}^{n+m}$ and matrices acting on them will be denoted in bold font (e.g. $\mathbf{A}$). Saddle points of $\varphi$ will be denoted $\mathbf{\bar{z}}=(\bar{x},\bar{y})\in\mathbb{R}^{n+m}$.
}

\subsection{\color{black}Dynamical systems}
\begin{definition}[Flows and semi-flows]\label{def:semiflows}
A triple $(\phi,X,\rho)$ is a flow (resp. semi-flow) if $(X,\rho)$ is a metric space, $\phi$ is a continuous map from $\mathbb{R}\times X$ (resp. $\mathbb{R}_+\times X$) to $X$ which satisfies the two properties
\begin{enumerate}[(i)]
\item For all $x\in X$, $\phi(0,x)=x$.
\item For all $x\in X$, $t,s\in\mathbb{R}$ (resp. $\mathbb{R}_+$),
\begin{equation}
\phi(t+s,x)=\phi(t,\phi(s,x)).
\end{equation}
\end{enumerate}
When there is no confusion over which (semi)-flow is meant, we shall denote $\phi(t,x(0))$ as $x(t)$. For sets $A\subseteq\mathbb{R}$ (resp. $\mathbb{R}_+$) and $B\subseteq X$ we define $\phi(A,B)=\{\phi(t,x):t\in A,x\in B\}$.
\end{definition}
\icl{We say that a trajectory $x(t)$ of a semi-flow converges to a trajectory $y(t)$ of the semi-flow if
\begin{gather}
\rho(x(t)-y(t))\to0 \ \text{as} \  t\to\infty. \label{eq:conSol}
\end{gather}
}
\vspace{-4mm}
\begin{definition}[Global convergence]
We say that a (semi)-flow $(\phi,X,\rho)$ is \textit{globally convergent}, if for all initial conditions $x\in X$, the trajectory $\phi(t,x)$ converges to \icl{an equilibrium point.}
\end{definition}

A specific form of incremental stability, which we will refer to as pathwise stability, will be needed in the analysis that follows.
\begin{definition}[Pathwise stability]\label{def:pathwise-stability}
We say that a semi-flow $(\phi,X,\rho)$ is {\color{black}pathwise stable}
if for any two trajectories $x(t),x'(t)$ the distance $\rho(x(t),x'(t))$ is non-increasing in time.
\end{definition}
\icl{Note that in a pathwise stable semi-flow $(\phi,X,\rho)$, for each $t\in\mathbb{R_+}$ that the map from $x\in X$ to $\phi(t, x )$ is non-expansive.}

As the subgradient method has a discontinuous vector field we need the notion of Carath\'eodory solutions of differential equations.
\begin{definition}[Carath\'eodory solution]\label{def:Caratheodory-solution}
We say that a trajectory $\mathbf{z}(t)$ is a \textit{Carath\'eodory solution} to a differential equation $\dot{\mathbf{z}}=\mathbf{f}(\mathbf{z})$, if $\mathbf{z}$ is an absolutely continuous function of $t$, and for almost all times $t$, the derivative $\dot{\mathbf{z}}(t)$ exists and is equal to $\mathbf{f}(\mathbf{z}(t))$.
\end{definition}

\section{Problem formulation}
\label{sec:problem_formulation}
The main object of study in 
{\color{black}Part I}
is the \emph{gradient method} on an arbitrary concave-convex function in $C^2$. 
\begin{definition}[Gradient method]
\label{gradmethod-definition}
Given $\varphi$ a $C^2$ concave-convex function on $\mathbb{R}^{n+m}$, we define the \textit{gradient method} as the flow on $(\mathbb{R}^{n+m},d)$, \icl{where $d$ denotes the Euclidean metric,} generated by the differential equation
\begin{equation}
\label{gradmethod-fullspace}
\begin{aligned}
\dot{x}&=\varphi_x,\\
\dot{y}&=-\varphi_y.
\end{aligned}
\end{equation}
\end{definition}
It is clear that the saddle points of $\varphi$ are exactly the {\color{black}equilibrium points of \eqref{gradmethod-fullspace}.}



In our companion paper \cite{Holding-Lestas-gradient-method-Part-II} we study instead the \emph{subgradient method} where the gradient method (\autoref{gradmethod-definition}) is restricted to a convex set $K$ by the addition of a projection term to the differential equation \eqref{gradmethod-fullspace}.
\begin{definition}[Subgradient method]\label{def:subgradient-method}
Given a non-empty closed convex set $K\subseteq\mathbb{R}^{n+m}$ and a $C^2$ function $\varphi$ that is concave-convex on $K$, we define the \textit{subgradient method on $K$} as a semi-flow on $(K,d)$ consisting of Carath\'eodory solutions of
\begin{equation}\label{gradmethod-convex-domain}
\begin{aligned}
\dot{\mathbf{z}}&=\mathbf{f}(\mathbf{z})-\mathbf{P}_{N_K(\mathbf{z})}(\mathbf{f}(\mathbf{z})),\\
\mathbf{f}(\mathbf{z})&=\icl{\begin{bmatrix}
                        \varphi_x\\ 
			-\varphi_y
                       \end{bmatrix}}.
\end{aligned}
\end{equation}
\end{definition}
{\color{black}Note that the gradient method is the} subgradient method on $\mathbb{R}^{n+m}$. \icl{It should also be noted that the subgradient method \eqref{gradmethod-convex-domain} can be seen as a projected dynamical system with on ODE given by $\mathbf{\dot z} = P_{T_{K(\mathbf{z})}}(\mathbf{f}(\mathbf{z}))$, where $T_{K(\mathbf{z})}$ is the tangent cone to $K$ at $\mathbf{z}$ (see \cite{brogliato2006equivalence} for various equivalent representations).}
{\color{black}In \appendixref{sec:the-addition-of-constant-gains} we
also consider the addition of constant gains to the gradient and subgradient method.}


We briefly summarise {\color{black} below the main contributions of this paper~(Part I).}
\begin{itemize}
\item We provide an exact {\color{black}characterization} of the limiting solutions of the gradient method \eqref{gradmethod-fullspace} applied to {\color{black}an} arbitrary concave-convex {\color{black}function} which is not assumed to be strictly concave-convex. Despite the non-linearity of the gradient dynamics, we show that these limiting solutions solve an explicit linear ODE given by derivatives of the concave-convex function at a {\color{black}saddle point.}
\item We show {\color{black}that the lack of convergence in gradient dynamics can lead to a problematic behaviour, \icl{in the sense that} arbitrarily small stochastic perturbations lead to an unbounded variance \icl{when the set of saddle points includes a bi-infinite line}.}
\item {\color{black}We provide an exact classification of the limiting solutions of the subgradient method on {affine subspaces} by extending the result described in the first bullet point. This will be important for the analysis of general subgradient dynamics considered in Part II \cite{Holding-Lestas-gradient-method-Part-II}. In particular, we show in Part II that the limiting behaviour of the subgradient method on arbitrary convex domains reduces to the limiting behaviour on affine subspaces.}

\end{itemize}

%
%
%

\section{Main Results}
\label{sec:Main}
\icl{This section includes the main results of the paper. Before \icl{presenting {those} some known results in the literature are stated.}}

\icl{An known property of the gradient method \eqref{gradmethod-fullspace} is the fact that it is pathwise stable, which is stated below as \autoref{incrementally-stable-fullspace}. This follows from the fact that the negative of the right-hand side in \eqref{gradmethod-fullspace} is maximal monotone \cite{rockafellar1971saddle}, \cite{goebel2017stability}.}


\begin{lemma}\label{incrementally-stable-fullspace}
Let $\varphi$ be $C^2$ and concave-convex on $\mathbb{R}^{n+m}$, then the gradient method \eqref{gradmethod-fullspace} is pathwise stable.
\end{lemma}

 \icl{Since} saddle points are equilibrium points of the gradient method \icl{the result below immediately follows\footnote{\icl{See also Figure~\ref{fig:twocircles-shaded} in \appendixref{sec:Geometry-of-S-bar-and-S} for a graphical illustration.}}.}
\begin{corollary}\label{cor:dist}
Let $\varphi$ be $C^2$ and concave-convex on $\mathbb{R}^{n+m}$, then the distance of a solution of \eqref{gradmethod-fullspace} to any saddle point is non-increasing in time.
\end{corollary}


\icl{Using LaSalle's theorem and \autoref{incrementally-stable-fullspace} we obtain \autoref{LaSalle-result-full-space}, which is proved in \appendixref{sec:Geometry-of-S-bar-and-S}.  Note that the notion of convergence to a solution in \autoref{LaSalle-result-full-space} (see its definition in \eqref{eq:conSol})
is stronger than that of convergence to a set.} 
\begin{corollary}\label{LaSalle-result-full-space}
Let $\varphi$ be $C^2$, concave-convex on $\mathbb{R}^{n+m}$ \icl{and have at least one saddle point. Then each solution of the} gradient method \eqref{gradmethod-fullspace} converges to a solution
of \eqref{gradmethod-fullspace} which has constant distance from any saddle point.
\end{corollary}

Thus classifying the limiting behaviour of the gradient method reduces to the problem of finding all solutions {\color{black}that lie} a constant distance from \textit{any} saddle point.
In order to facilitate the presentation of the results, for a given concave-convex function $\varphi$ we define the following sets:
\begin{itemize}
\item $\bar{\mathcal{S}}$ will denote the set of saddle points of $\varphi$.
\item $\mathcal{S}$ will denote the set of solutions to \eqref{gradmethod-fullspace} that are a constant distance from any saddle point of $\varphi$.
\end{itemize}

Note that if $\bar{\mathcal{S}}=\mathcal{S}\ne\emptyset$ then \autoref{LaSalle-result-full-space} gives the convergence of the gradient method to 
a saddle point.

Our first main result is that solutions of the gradient method converge to solutions that satisfy an explicit linear ODE.

To present our results we define the following matrices of partial derivatives of $\varphi$
\begin{equation}\label{def-of-A-and-B}
\begin{aligned}
\mathbf{A}(\mathbf{z})&=\begin{bmatrix}
0&\varphi_{xy}(\mathbf{z})\\
-\varphi_{yx}(\mathbf{z})&0
\end{bmatrix}\\
\mathbf{B}(\mathbf{z})&=\begin{bmatrix}
\varphi_{xx}(\mathbf{z})&0\\
0&-\varphi_{yy}(\mathbf{z})
\end{bmatrix}.
\end{aligned}
\end{equation}
For simplicity of notation we shall state the result for $\mathbf{0}\in\bar{\mathcal{S}}$; the general case may be obtained by a translation of coordinates.
\begin{theorem}\label{Classification-in-terms-of-B(z)z}
Let $\varphi$ be $C^2$ and concave-convex on $\mathbb{R}^{n+m}$. Let $\mathbf{0}\in\bar{\mathcal{S}}$ then \icl{all} solutions in $\mathcal{S}$ solve the linear ODE: 
\begin{equation}\label{eq:limiting-ODE-fullspace}
\dot{\mathbf{z}}(t)=\mathbf{A}(\mathbf{0})\mathbf{z}(t).
\end{equation}
Furthermore, a solution $\mathbf{z}(t)$ to \eqref{eq:limiting-ODE-fullspace} is in $\mathcal{S}$ if and only if for all $t\in\mathbb{R}$ and $r\in[0,1]$,
\begin{equation}\label{eq:first-thm-kernel-condition}
\mathbf{z}(t)\in\ker(\mathbf{B}(r\mathbf{z}(t)))\cap \ker(\mathbf{A}(r\mathbf{z}(t))-\mathbf{A}(\mathbf{0}))
\end{equation}
where $\mathbf{A}(\mathbf{z})$ and $\mathbf{B}(\mathbf{z})$ are defined by \eqref{def-of-A-and-B}.
\end{theorem}
\icl{The proof of \autoref{Classification-in-terms-of-B(z)z} is provided in \appendixref{sec:Geometry-of-S-bar-and-S} and \appendixref{sec:Classification-of-S}.}
{\color{black}The significance of this result is discussed in the
remarks below.}
\begin{remark}
{\color{black}It should be noted that despite} the \emph{non-linearity} of the gradient dynamics \eqref{gradmethod-fullspace},
the limiting solutions solve {\color{black}a \emph{linear}} ODE with explicit coefficients depending only on the derivatives of $\varphi$ at the saddle point.
\end{remark}
\begin{remark}
\color{black}An important consequence of this exact characterisation of the limiting behaviour, is the fact that the problem of proving global {\color{black}convergence to a saddle point is reduced to that of showing that there are no non-trivial limiting solutions.}
\end{remark}
\begin{remark}\label{rem:check}
{\color{black} Condition} \eqref{eq:first-thm-kernel-condition} appears to be very hard to check, as it requires knowledge of the trajectory for all times $t\in\mathbb{R}$. However,
{\color{black}when the aim is to prove convergence to an equilibrium point,
the form of condition \eqref{eq:first-thm-kernel-condition} makes the stability condition more powerful,} 
as it makes it easier to prove that {\color{black}non-trivial} trajectories do \emph{not} satisfy the condition. \icl{In particular, for various classes of gradient dynamics, the structure of matrices $\mathbf{A}(\mathbf{z})$ and $\mathbf{B}(\mathbf{z})$ are often sufficient to deduce that \eqref{eq:limiting-ODE-fullspace}, \eqref{eq:first-thm-kernel-condition} are only satisfied by saddle points,  without explicitly knowing those a priory, thus proving global convergence to a saddle point. Such examples will be discussed in part II of this manuscript, and include modifications that ensure global convergence to a saddle point without having to resort to an assumption of a {\em strictly} concave-convex function. An application of such a modification to the problem of multipath routing will also be discussed in part II.}
\end{remark}
\begin{remark}[Localisation]\label{rem:localisation}
{\color{black}The conditions in the Theorem use
only local} information about the concave-convex function $\varphi$, in the sense that if $\varphi$ is only concave-convex on a convex subset $K\subseteq\mathbb{R}^{n+m}$ which contains $\mathbf{0}$, then any trajectory $\mathbf{z}(t)$ of the gradient method \eqref{gradmethod-fullspace} that lies a constant distance from any saddle point in $K$ and 
{\color{black}does not leave $K$ at any time $t$} will obey the conditions of the theorem.
\icl{\begin{remark}
A main significance of the exact characterization of the limiting behaviour of gradient method in Theorem \ref{Classification-in-terms-of-B(z)z} is the fact that it has a natural generalization to subgradient dynamics on affine subspaces (presented in section \ref{sec:affine}), which can be used to characterize the limiting solutions of the subgradient method. The latter is a more involved problem due to the non-smoothness of the dynamics and is addressed in part II of this manuscript.
\end{remark}}
\icl{
\begin{remark}\label{rem:localglobal}
One of the results proved in \appendixref{sec:Classification-of-S}
is the fact that the set $\mathcal{S}$ is convex (Proposition\footnote{\icl{This result is also generalized in part II to pathwise stable semiflows.}} \ref{S-is-convex}). From this it follows that global convergence of the gradient method can be deduced from only local convergence properties about one of the saddle points. This is stated below as \autoref{cor:local} which is proved in \appendixref{sec:Classification-of-S}.
\vspace{-1mm}
\begin{lemma}\label{cor:local}
Let $\varphi$ be $C^2$, concave-convex on $\mathbb{R}^{n+m}$, and let  $\mathbf{\bar{z}}$ be a saddle point. Then the gradient method \eqref{gradmethod-fullspace} is globally convergent if and only if there exists a neighbourhoud $\mathcal{N}$ of $\mathbf{\bar{z}}$ such that trajectories with initial condition in $\mathcal{N}$ converge to a saddle point.
\end{lemma}

\end{remark}}
\end{remark}


As a simple illustration of the use of \autoref{Classification-in-terms-of-B(z)z} we show how to recover the well known result that the gradient method is globally convergent under the assumption that $\varphi$ is strictly concave-convex \icl{(first part of \autoref{rem:gradmethod-strict-case})}. \icl{It is also known that it is sufficient to have the strictness only locally  \cite{cherukuri2016role} 
which can be recovered from \autoref{cor:local} (second part of \autoref{rem:gradmethod-strict-case}). Note that in general strictness is not needed to deduce convergence and such examples will be discussed in part II (see also \autoref{rem:check}).
}

\begin{example}\label{rem:gradmethod-strict-case} $\ $
\begin{enumerate}
\item \label{ex:strict1}
Suppose $\varphi$ is strictly concave (the strictly convex case is similar), then $\varphi_{xx}$ is of full rank except at isolated points, and the condition \eqref{eq:first-thm-kernel-condition} can only hold if $x(t)=0$. Then the ODE \eqref{eq:limiting-ODE-fullspace} implies that $y(t)$ is constant, and hence $(x(t),y(t))$ is a saddle point. Thus the only limiting solution of the gradient method \eqref{gradmethod-fullspace} are the saddle points, which establishes that it is globally convergent.
\item \icl{Suppose $\varphi$ is concave-convex, but it is strictly concave-convex only in an open region $\mathcal{B}\subset\mathbb{R}^{n+m}$ that includes a saddle point $\mathbf{\bar{z}}$, rather than in the whole of  $\mathbb{R}^{n+m}$.
     From the stability of $\mathbf{\bar{z}}$ there exists a neighbourhood $\mathcal{N}\subset\mathcal{B}$ of $\mathbf{\bar{z}}$ such that trajectories that start in $\mathcal{N}$ will remain in $\mathcal{B}$ for all times. Furthermore, using the arguments in \ref{ex:strict1}) one can deduce that region $\mathcal{N}$ does not include a trajectory $\mathbf{z}\in\mathcal{S}$ that is not a saddle point. Therefore global convergence can be deduced from \autoref{cor:local}.}

\end{enumerate}
\end{example}

From \autoref{Classification-in-terms-of-B(z)z} we deduce some further results that give a more easily understandable classification of the limiting solutions of the gradient method for simpler forms of $\varphi$.

In particular, the `linear' case occurs when $\varphi$ is a quadratic function, as then the gradient method \eqref{gradmethod-fullspace} is a linear system of ODEs. In this case $\mathcal{S}$ has a simple explicit form in terms of the Hessian matrix of $\varphi$ at $\mathbf{0}\in\bar{\mathcal{S}}$, and in general this provides an inclusion as described below, which can be used to prove global convergence of the gradient method using only local analysis at a saddle point.
\begin{theorem}\label{S-inside-Slinear}
Let $\varphi$ be $C^2$, concave-convex on $\mathbb{R}^{n+m}$ and $\mathbf{0}\in\bar{\mathcal{S}}$. Then define
\begin{equation}\label{equation-Slinear}
\!\!\!\mathcal{S}_{\text{linear}}\!=\!\Span\{\mathbf{v}\!\in\!\ker(\mathbf{B})\!:\!\mathbf{v}\text{ is an eigenvector of }\mathbf{A}\}
\end{equation}
where $\mathbf{A}=\mathbf{A}(\mathbf{0})$ and $\mathbf{B}=\mathbf{B}(\mathbf{0})$ in \eqref{def-of-A-and-B}.
Then $\mathcal{S}\subseteq\mathcal{S}_{\text{linear}}$ with equality if $\varphi$ is a quadratic function.
\end{theorem}
{\icl{The proof of \autoref{S-inside-Slinear} is provided in \appendixref{sec:Classification-of-S}.}
Here we draw an analogy with the recent study \cite{Bai2010} on the discrete time gradient method in the quadratic case. There the gradient method is proved to be semi-convergent if and only if $\ker(\mathbf{B})=\ker(\mathbf{A}+\mathbf{B})$, i.e. if $\mathcal{S}_{\text{linear}}\subseteq\mathcal{\bar{S}}$. \autoref{S-inside-Slinear} includes a continuous time version of this statement.

{\color{black}We next consider the effect of noise when oscillatory solutions occur, and show that arbitrarily small stochastic perturbations lead to an unbounded variance \icl{when the set of saddle points includes a bi-infinite line}.
In particular, we consider the addition of white noise}
 to the dynamics \eqref{gradmethod-fullspace}. This {\color{black}
leads to the} following stochastic differential equations
\begin{equation}\label{eq:noisy-gradient-method-fullspace}
\begin{aligned}
dx(t)=\varphi_xdt+\Sigma^xdB^x(t)\\
dy(t)=-\varphi_ydt+\Sigma^ydB^y(t)
\end{aligned}
\end{equation}
where $B^x(t),B^y(t)$ are independent standard Brownian motions in $\mathbb{R}^n,\mathbb{R}^m$ respectively, and $\Sigma^x,\Sigma^y$ are positive definite symmetric matrices in $\mathbb{R}^{n\times n},\mathbb{R}^{m\times m}$ respectively.
\begin{theorem}\label{prop:unstable-to-noise}
Let $\varphi\in C^2$ be concave-convex on $\mathbb{R}^{n+m}$. Let $\mathbf{0}\in\bar{\mathcal{S}}$ and $\mathcal{S}$ contain a bi-infinite line. Consider the noisy dynamics \eqref{eq:noisy-gradient-method-fullspace}. Then, for any initial condition, the variance of the solution tends to infinity as $t\to\infty$, in that
\begin{equation}\label{eq:variance-growth}
\mathbb{E}|\mathbf{z}(t)|^2\to\infty  \text{ as }t\to\infty.
\end{equation}
where $\mathbb{E}$ denotes the expectation operator.
\end{theorem}
{\icl{The proof of \autoref{prop:unstable-to-noise} is provided in \appendixref{sec:Classification-of-S}.}
The condition that $\mathcal{S}$ contains a bi-infinite line is satisfied, for example, {\color{black}if the set $\mathcal{S}$ is not just a} single point {\color{black}and} $\varphi$ is a quadratic function, and {\color{black}can occur} in applications, e.g. in the multi-path routing example given in our companion paper \cite{Holding-Lestas-gradient-method-Part-II}.

One of the main applications of the gradient method is \icl{to provide convergence to a saddle point of a Lagrangian following from a} 
concave optimization problem where some of the constraints are relaxed by Lagrange multipliers. When all the relaxed constraints are linear, {\color{black}the Lagrangian} $\varphi$ has the form
\begin{equation}\label{varphi-relaxed-linear-constraints}
\varphi(x,y)=U(x)+y^T(Dx+e)
\end{equation}
{\color{black}where $U(x)$ is a concave cost function, $y$ are the Lagrange multipliers, and $D, e$ are a constant matrix and vector respectively associated with the equality constraints}. Under the assumption that $U$ is analytic we obtain a simple exact characterisation of $\mathcal{S}$. One specific case of this was studied by the authors previously in \cite{Holding-Lestas-CDC2013}, but without the analyticity condition.
\begin{theorem}\label{Classification-of-S-in-relaxed-linear-constraints-case}
Let $\varphi$ be defined by \eqref{varphi-relaxed-linear-constraints} with $U$ analytic and $D\in\mathbb{R}^{m\times n}$,  $e\in\mathbb{R}^m$ constant. Assume that $(\bar{x},\bar{y})=\mathbf{\bar{z}}$ is a saddle point of $\varphi$. Then $\mathcal{S}$ is given by
\begin{align}\label{equation-S-relaxed-linear-constraints}
\!\!\!\!\!\mathcal{S}&=\mathbf{\bar{z}}+\Span\{(x,y) \in \mathcal{W}\times\mathbb{R}^m :(x,y)\text{ is}\notag\\
&\quad\quad\quad\quad\quad\quad\text{ an eigenvector of }\begin{bmatrix}
0&D^T\\
-D&0
\end{bmatrix}\Big\}\\
\!\!\!\mathcal{W}&=
\{x\in\mathbb{R}^{n}:s\mapsto U(sx+\bar{x})\text{ is linear for }s\in\mathbb{R}\}.\notag
\end{align}
Furthermore $\mathcal{W}$ is an affine subspace.
\end{theorem}
{\icl{The proof of \autoref{S-inside-Slinear} is provided in \appendixref{sec:Classification-of-S}.}
\icl{
\begin{remark}\label{rem:simpleChar}
It should be noted that a simple characterization of $\mathcal{S}$ as in \autoref{S-inside-Slinear} and \autoref{Classification-of-S-in-relaxed-linear-constraints-case} is not always necessary in order to prove global convergence to a saddle point by means of \autoref{Classification-in-terms-of-B(z)z}. For example, in the modification methods discussed in part II the structure of the matrices $\mathbf{A}(\mathbf{z})$ and $\mathbf{B}(\mathbf{z})$ are sufficient to deduce global convergence.
\end{remark}
}

\subsection{The subgradient method on affine subspaces}\label{sec:affine}
We now extend the exact classification (\autoref{Classification-in-terms-of-B(z)z}) to the subgradient method on affine subspaces.
{\color{black}The significance of this result is that it allows to provide a characterization of the limiting behaviour of the subgradient method in any convex domain. In particular, one of the main results that will be proved in Part II of this work is the fact that {\em the limiting behaviour of the subgradient method on a general convex domain, \icl{when an equilibrium point exists}, are solutions to subgradient dynamics on only affine subspaces}.


In order to consider subgradient dynamics on an affine subspace, we let} $V$ be an affine subspace of $\mathbb{R}^{n+m}$ and let $\mathbf{\Pi}\in\mathbb{R}^{(n+m)^2}$ be the orthogonal projection matrix onto the orthogonal complement of the normal cone $N_V$. Then the subgradient method \eqref{gradmethod-convex-domain} on $V$ is given by
\begin{align}
\dot{\mathbf{z}}&\icl{=\mathbf{f}(\mathbf{z})-\mathbf{P}_{N_V}(\mathbf{f}(\mathbf{z}))} \nonumber\\
&=\mathbf{\Pi}\mathbf{f}(\mathbf{z}) \label{eq:projected-gradient-method}
\end{align}
where $\mathbf{f}(\mathbf{z})=
\begin{bmatrix}
                        \varphi_x&\!\!\!
			-\varphi_y
                       \end{bmatrix}^T$. We generalise \autoref{Classification-in-terms-of-B(z)z} for this projected form of the gradient method. As with the statement of \autoref{Classification-in-terms-of-B(z)z}, we state the result for $\mathbf{0}$ being an equilibrium point; the general case may be obtained by a translation of coordinates.
\begin{theorem}\label{thm:projected-gradient-method-result}
Let $\mathbf{\Pi}\in\mathbb{R}^{(n+m)^2}$ be an orthogonal projection matrix, $\varphi$ be $C^2$ and concave-convex on $\mathbb{R}^{n+m}$, and $\mathbf{0}$ be an equilibrium point of \eqref{eq:projected-gradient-method}. Then the trajectories $\mathbf{z}(t)$ of \eqref{eq:projected-gradient-method} that lie a constant distance from any equilibrium point of \eqref{eq:projected-gradient-method} are exactly the solutions to the linear ODE:
\begin{equation}\label{eq:projected-gradient-method-linear-ODE}
\dot{\mathbf{z}}(t)=\mathbf{\Pi}\mathbf{A}(\mathbf{0})\mathbf{\Pi}\mathbf{z}(t)
\end{equation}
that satisfy, for all $t\in\mathbb{R}$ and $r\in[0,1]$, the condition
\begin{equation}\label{eq:first-thm-kernel-condition-projected-gradient-method}
\!\!\mathbf{z}(t)\in\ker(\mathbf{\Pi}\mathbf{B}(r\mathbf{z}(t))\mathbf{\Pi})\cap \ker(\mathbf{\Pi}(\mathbf{A}(r\mathbf{z}(t))-\mathbf{A}(\mathbf{0}))\mathbf{\Pi})
\end{equation}
where $\mathbf{A}(\mathbf{z})$ and $\mathbf{B}(\mathbf{z})$ are {\color{black}defined by \eqref{def-of-A-and-B}.}
\end{theorem}
{\icl{The proof of \autoref{thm:projected-gradient-method-result} is provided in \appendixref{sec:Classification-of-S}.}

{\color{black}\section{Applications}
In many applications associated with saddle point problems, the variables need to be constrained in prescribed domains. These include, for example, positivity constraints on dual variables  in optimization problems where some of the inequality constraints are relaxed with Lagrange multipliers, or more general convex constraints on primal variables. Therefore  applications will be studied in Part II of this work where subgradient dynamics will be \icl{analyzed\footnote{\icl{It should be noted that there are also classes of constrained optimization problems that can also be solved by means of smooth saddle point dynamics, such as the dynamics proposed in \cite{dorr2012smooth}.}}. }

It should be noted, that apart from their significance for saddle point problems without constraints\footnote{\color{black}Note that these include also dual versions of optimization problems with equality constraints.}, a main significance of the results in Part I is that they also lead to a characterization of the asymptotic behaviour of subgradient dynamics. In particular, as mentioned in section \ref{sec:affine}, it will be proved in Part II of this work that the asymptotic behaviour of subgradient dynamics on a general convex domain, is given by solutions to subgradient dynamics on only affine subspaces, which is a class of dynamics the asymptotic behaviour of which can be exactly determined using the results in Part I.
}

\section{Conclusion}
\label{sec:Conclusions}
{\color{black}
We have considered in {\color{black}Part I} the problem of convergence to a saddle point of a general concave-convex function, which is not necessarily strictly concave-convex, via gradient dynamics.
We have provided an exact characterization to the asymptotic behaviour of such dynamics, and have shown that despite their nonlinearity, convergence is guaranteed to trajectories that {\color{black}satisfy} an explicit {\em linear} ODE. We have also shown that when convergence to a saddle point is not {\color{black}ensured} then the behaviour of such dynamics can be problematic, with arbitrarily small noise leading to an unbounded variance \icl{when the set of sadddle points includes a bi-infinite line}. These results have also been extended to subgradient dynamics on affine subspaces, where an exact characterization of their asymptotic behaviour as linear ODEs has also been derived. This class of dynamics will be used as a basis for the results in Part II. In particular, it will be shown in Part II that subgradient dynamics on a general convex domain \icl{that have an equilibrium point,} have an $\omega$-limit set that consists of trajectories that are solutions to subgradient dynamics on only affine subspaces. Various examples and applications will also be presented in Part II.
}


\bibliography{Paper_2}
\bibliographystyle{plain}
\appendix
\section{}

\label{sec:proofs}
In appendices \ref{sec:Geometry-of-S-bar-and-S} and \ref{sec:Classification-of-S} we prove the main results of the paper which are stated in \autoref{sec:Main}.

We first give a brief outline of the derivations of the results to improve {\color{black}their readability.} 
Before we give this summary we \icl{define} some additional notation.

Given $\mathbf{\bar{z}}\in\bar{\mathcal{S}}$, we denote the set of solutions to the gradient method \eqref{gradmethod-fullspace} that are a constant distance from $\mathbf{\bar{z}}$, (but not necessarily other saddle points), as $\mathcal{S}_{\mathbf{\bar{z}}}$. It is later proved that $\mathcal{S}_{\mathbf{\bar{z}}}=\mathcal{S}$ but until then the distinction is important.


{\em Outline of Proofs}:
\begin{itemize}
\item First in \appendixref{sec:Geometry-of-S-bar-and-S} we use the pathwise stability of the gradient method (\autoref{incrementally-stable-fullspace}) and geometric arguments to establish convexity properties of $\mathcal{S}$. 
\icl{\autoref{infinite-line-of-saddles} states that 
$\bar{\mathcal{S}}$} 
can only contain bi-infinite lines in degenerate cases. \autoref{S-orthogonal-to-barS-fullspace} gives an orthogonality condition between $\mathcal{S}$ and $\bar{\mathcal{S}}$ which roughly says that the larger $\bar{\mathcal{S}}$ is, the smaller $\mathcal{S}$ is. These allow us to prove the key result of the section, \autoref{combined-convexity-of-S-and-bar-S-fullspace}, which states that any convex combination of $\mathbf{\bar{z}}\in\bar{\mathcal{S}}$ and $\mathbf{z}(t)\in\mathcal{S}_{\mathbf{\bar{z}}}$ lies in $\mathcal{S}_{\mathbf{\bar{z}}}$.

\item In \appendixref{sec:Classification-of-S} we use the geometric results of \appendixref{sec:Geometry-of-S-bar-and-S} to prove {\color{black}Theorems~\ref{Classification-in-terms-of-B(z)z},  \ref{S-inside-Slinear}.}

To prove \autoref{prop:unstable-to-noise} we first prove \icl{\autoref{lem:conserved-quantity-for-noise-proof}} (analogous to \autoref{infinite-line-of-saddles}) that tells us that $\mathcal{S}$ containing a bi-infinite line implies the presence of a quantity conserved by all solutions of the gradient dynamics \eqref{gradmethod-fullspace}. In the presence of noise, the variance of this quantity converges to infinity and allows us to prove \autoref{prop:unstable-to-noise}.

To prove \autoref{Classification-of-S-in-relaxed-linear-constraints-case} we construct a quantity $V(\mathbf{z})$ that is conserved by solutions in $\mathcal{S}$. In the case considered this has a natural interpretation in terms of the utility function $U(x)$ and the constraints $g(x)$.

Finally \autoref{thm:projected-gradient-method-result} is proved by modifying the {\color{black}proof of Theorem~\ref{Classification-in-terms-of-B(z)z}} to take into account the addition of the projection matrix.
\end{itemize}

\appendices
\section{Geometry of \texorpdfstring{$\bar{\mathcal{S}}$ and $\mathcal{S}$}{S and S}}
\label{sec:Geometry-of-S-bar-and-S}
 In this section we will use the gradient method to derive geometric properties of convex-concave functions. We will start with some simple results which are then used as a basis to derive \autoref{combined-convexity-of-S-and-bar-S-fullspace} the main result of this \icl{section.}

 \icl{We first provide a proof for \autoref{LaSalle-result-full-space} associated with the convergence of all solutions of the gradient method to solutions in the set $\mathcal{S}$.
 \begin{proof}[Proof of \autoref{LaSalle-result-full-space}]
 Using Lasalle's invariance principle with $|\mathbf{z}(t)-\mathbf{\bar z}|^2$ as the Lyapunov like function, where $\mathbf{\bar z}$ is a saddle point, we get convergence of all solutions of the gradient method to the set of solutions that are a constant distance from all saddle points (denoted as set $\mathcal{S}$). In the remainder of the proof we 
 strengthen the convergence to the set $\mathcal{S}$, to convergence to a solution\footnote{\icl{The proof is based on arguments analogous to those in \cite[Proposition 42]{Holding-Lestas-CDC2013}}} in $\mathcal{S}$ (see the definition of convergence to a solution in \eqref{eq:conSol}).

 Let $\mathbf{z}(t)$ be a solution of \eqref{gradmethod-fullspace}.
     From the convergence to the set $\mathcal{S}$ and \autoref{incrementally-stable-fullspace} there exist points $\mathbf{z}^{(n)}\in\mathcal{S}$ and times $t_n$ such that,
 \begin{eqnarray}
 |\mathbf{z}(t_n)-\mathbf{z}^{(n)} |\le 1/n.
 \end{eqnarray}
We now consider the solutions $\mathbf{z}^{(n)}(t)\in\mathcal{S}$ with $\mathbf{z}^{(n)}(t_n)=\mathbf{z}^{(n)}$. By an application of \autoref{incrementally-stable-fullspace}, we have for all $t\ge t_n$,
 \begin{eqnarray}\label{eq:zn}
|\mathbf{z}(t)-\mathbf{z}^{(n)}(t)|\le 1/n.
 \end{eqnarray}
 From the boundedness of $\mathbf{z}(t)$ and \eqref{eq:zn} the set and $\{\mathbf{z}^{(n)}:n\in\mathbb{N}\}$ is relatively compact, and by the constant distance of each trajectory in $\mathcal{S}$ from all saddle points, the set of initial conditions $\{\mathbf{z}^{(n)}(0):n\in\mathbb{N}\}$ is also relatively compact. There is hence a subsequence $n_k$ for which $\mathbf{z}^{(n_k)}(0)\to\mathbf{z}'(0)\in\mathcal{S}$ as $k\to\infty$. We claim that $|\mathbf{z}(t)-\mathbf{z}'(t)|\to0$ as $t\to\infty$. Indeed, for any $\epsilon>0$ there exists a $k\in\mathbb{N}$ such that for all $t\ge t_{n_k}$, we have
\begin{eqnarray}
  |\mathbf{z}(t)-\mathbf{z}^{(n_k)}(t)|\le \varepsilon/2
  \end{eqnarray}
  and also for all $t\ge0$,
   \begin{eqnarray}
   |\mathbf{z}'(t)-\mathbf{z}^{(n_{k})}(t)|\le \varepsilon/2
   \end{eqnarray}
   where in each case we have used \autoref{incrementally-stable-fullspace}.
   The claim now follows from the triangle inequality, which completes the proof of \autoref{LaSalle-result-full-space}.
 \end{proof}}

\begin{lemma}[\icl{\cite{rockafellar1970convex}}]\label{convexity-of-saddles-fullspace}
Let $\varphi\in C^2$ be concave-convex on $\mathbb{R}^{n+m}$, then $\bar{\mathcal{S}}$, the set of saddle {\color{black}points} of $\varphi$, is closed and convex.
\end{lemma}

\begin{figure}[h]
\centering
\begin{tikzpicture}[scale=0.6,anchor=mid,>=latex',line join=bevel,]
\coordinate (L) at (0,0);
\coordinate (R) at (10,0);
\coordinate (nearest) at (5,0);
\coordinate [label=above:\textcolor{blue}{$\mathbf{z}$}] (z) at (5,1);
\coordinate [label=below:\textcolor{blue}{$\mathbf{\bar{a}}$}] (a) at ($(nearest)-(1,0)$);
\coordinate [label=below:\textcolor{blue}{$\mathbf{\bar{b}}$}] (b) at ($(nearest)+(1,0)$);
\begin{scope} 
 \clip (a) circle ({sqrt(2)});
 \fill[green!20] (b) circle ({sqrt(2)});
\end{scope}
\fill (z) circle (3pt);
\draw [<->] (L) -- (nearest) -- (R) node[right,fill=white] {$L$};
\draw [dashed,<->] ($(nearest)-(0,3)$) -- ($(nearest)+(0,3)$);

\fill (a) circle (3pt);
\fill (b) circle (3pt);

\draw (a) circle ({sqrt(2)});
\draw (b) circle ({sqrt(2)});

\end{tikzpicture}
\caption{$\mathbf{\bar{a}}$ and $\mathbf{\bar{b}}$ are two saddle points of $\varphi$ which is $C^2$ and concave-convex on $\mathbb{R}^{n+m}$. Solutions of \eqref{gradmethod-fullspace} are constrained to lie in the shaded region for all positive time by \autoref{incrementally-stable-fullspace}.}\label{fig:twocircles-shaded}
\end{figure}
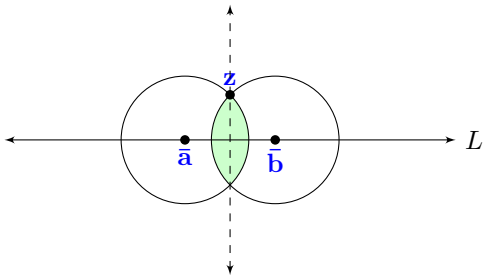
\begin{figure}[h]
\centering
\begin{tikzpicture}[scale=0.6,anchor=mid,>=latex',line join=bevel,]
\begin{scope}[
        yshift=-160,every node/.append style={
        yslant=0.5,xslant=-1.3},yslant=0.5,xslant=-1.3
                  ]
\draw[black,very thick] (0,0) rectangle (4,4);
\draw [help lines] (0,0) grid +(4,4);
\coordinate (bottomcentre) at (2,2);
\coordinate (zbottom) at (1.4,3.7);
\coordinate (zbottomend) at (0.4,2.5);
\fill[blue] (zbottom) circle (3pt);
\draw[blue,->] (zbottom) parabola (zbottomend);
\end{scope}
\begin{scope}[
            yshift=-100,every node/.append style={
            yslant=0.5,xslant=-1.3},yslant=0.5,xslant=-1.3
            ]
\fill[white,fill opacity=.75] (0,0) rectangle (4,4); 
\draw[black,very thick] (0,0) rectangle (4,4);
\draw [help lines] (0,0) grid +(4,4);
\coordinate (topcentre) at (2,2);
\coordinate (ztop) at (1.4,3.7);
\coordinate (ztopend) at (0.4,2.5);
\fill[blue] (ztop) circle (3pt);
\draw[blue,->] (ztop) parabola (ztopend);
\end{scope}
\fill (topcentre) circle (3pt);
\fill (bottomcentre) circle (3pt);
\draw [<->] ($(topcentre) ! 2 ! (bottomcentre)$) -- ($(bottomcentre) ! 2 ! (topcentre)$) node[above] {$L$};
\draw[red] [<->] (zbottom) -- (ztop);
\draw[red] [<->] (zbottomend) -- (ztopend);
\coordinate [label=right:\textcolor{blue}{$\mathbf{z}$}] (zbottomlabel) at (zbottom);
\coordinate [label=right:\textcolor{blue}{$\mathbf{z}+s\mathbf{b}$}] (ztoplabel) at (ztop);
\begin{scope}[
            yshift=-100,every node/.append style={
            yslant=0.5,xslant=-1.3},yslant=0.5,xslant=-1.3
            ]
\draw[black,very thick] (0,0) rectangle (4,4);
\end{scope}
\end{tikzpicture}
\caption{$L$ is a line of saddle points of $\varphi$ which is $C^2$ and concave-convex on $\mathbb{R}^{n+m}$. Solutions of \eqref{gradmethod-fullspace} starting on hyperplanes normal to $L$ are constrained to lie on these planes for all time. $\mathbf{z}$ lies on one normal hyperplane, and $\mathbf{z}+s\mathbf{b}$ lies on another. Considering the solutions of \eqref{gradmethod-fullspace} starting from each we see that by \autoref{incrementally-stable-fullspace} the distance between these two solutions must be constant and equal to $|s\mathbf{b}|$. %
}\label{fig:twoplanes}
\end{figure}
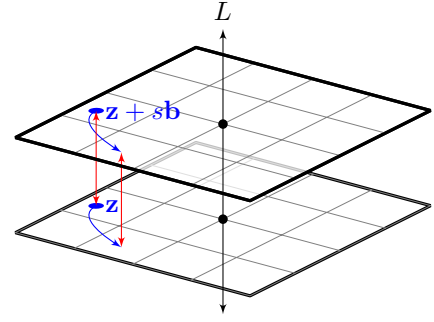
\begin{lemma}\label{infinite-line-of-saddles}
Let $\varphi$ be $C^2$ and concave-convex on $\mathbb{R}^{n+m}$. Let the set of saddle points of $\varphi$ contain the infinite line $L=\{\mathbf{a}+s\mathbf{b}:s\in\mathbb{R}\}$ for some $\mathbf{a},\mathbf{b}\in\mathbb{R}^{n+m}$. Then $\varphi$ is translation invariant in the direction of $L$, i.e. $\varphi(\mathbf{z})=\varphi(\mathbf{z}+s\mathbf{b})$ for any $s\in\mathbb{R}$.
\end{lemma}
\begin{proof}
We do this in two steps. First we will prove that the motion of the gradient method is restricted to linear manifolds normal to $L$. Let $\mathbf{z}$ be a point and consider the motion of the gradient method starting from $\mathbf{z}$. As illustrated in \autoref{fig:twocircles-shaded} we pick two saddle points $\mathbf{\bar{a}},\mathbf{\bar{b}}$ on $L$, then by \autoref{incrementally-stable-fullspace} the motion starting from $\mathbf{z}$ is constrained to lie in the (shaded)
region, which is the intersection of the two closed balls about $\mathbf{\bar{a}}$ and $\mathbf{\bar{b}}$ which have $\mathbf{z}$ on their boundaries. The intersection of
\icl{all such regions generated by a sequence of pairs of saddle points on $L$ each tending to infinity on opposite directions along $L$, is contained in the linear manifold normal to $L$ that contains $\mathbf{z}$.}

Next we claim that for $s\in\mathbb{R}$ the motion starting from $\mathbf{z}+s\mathbf{b}$ is exactly the motion starting from $\mathbf{z}$ shifted by $s\mathbf{b}$. As illustrated in \autoref{fig:twoplanes}, by \autoref{incrementally-stable-fullspace} the motion from $\mathbf{z}+s\mathbf{b}$ must stay a constant distance $s|\mathbf{b}|$ from the motion from $\mathbf{z}$. This uniquely identifies the motion from $\mathbf{z}+s\mathbf{b}$ and proves the claim.
Finally we deduce the full result by noting that the second claim  implies that $\varphi$ is defined up to an additive constant on each linear manifold as the motion of the gradient method contains all the information about the derivatives of $\varphi$. As $\varphi$ is constant on $L$, the proof is complete.
\end{proof}
We now use these techniques to prove orthogonality results about solutions in $\mathcal{S}$.
\begin{lemma}\label{S-orthogonal-to-barS-fullspace}
Let $\varphi\in C^2$ be concave-convex on $\mathbb{R}^{n+m}$, and $\mathbf{z}$ be a trajectory in $\mathcal{S}$, then $\mathbf{z}(t)\in M_{\bar{\mathcal{S}}}(\mathbf{z}(0))$ for all $t\in\mathbb{R}$, \icl{where $M_{\bar{\mathcal{S}}}(\mathbf{z}(0))$ denotes the manifold defined in \eqref{eq:orth_lin_man}.}
\end{lemma}
\begin{proof}
If $\bar{\mathcal{S}}=\{\mathbf{\bar{z}}\}$ or %
$\emptyset$ %
the claim is trivial. Otherwise we let $\mathbf{\bar{a}}\ne\mathbf{\bar{b}}\in\bar{\mathcal{S}}$ be arbitrary, and consider the spheres \icl{centred at $\mathbf{\bar{a}}$ and $\mathbf{\bar{b}}$, respectively, that 
each have the point $\mathbf{z}(t)$ on their boundary}. By \autoref{incrementally-stable-fullspace}, $\mathbf{z}(t)$ is constrained to lie on the intersection of these two spheres which lies inside $M_{L}(\mathbf{z}(0))$ where $L$ is the line segment between $\mathbf{\bar{a}}$ and $\mathbf{\bar{b}}$. As $\mathbf{\bar{a}}$ and $\mathbf{\bar{b}}$ were arbitrary this proves the lemma.
\end{proof}
\begin{lemma}\label{Orthogonal-and-constant-distance-from-one-implies-in-S}
Let $\varphi$ be $C^2$ and concave-convex on $\mathbb{R}^{n+m}$, $\mathbf{\bar{z}}\in\bar{\mathcal{S}}$ and $\mathbf{z}(t)\in\mathcal{S}_{\mathbf{\bar{z}}}$ lie in $M_{\bar{\mathcal{S}}}(\mathbf{z}(0))$ for all $t$. Then $\mathbf{z}(t)\in\mathcal{S}$.
\end{lemma}
\begin{proof}
If $\bar{\mathcal{S}}=\{\mathbf{\bar{z}}\}$ the claim is trivial. Let $\mathbf{\bar{a}}\in\bar{\mathcal{S}}\setminus\{\mathbf{\bar{z}}\}$ be arbitrary. Then by \autoref{convexity-of-saddles-fullspace} the line segment $L$ between $\mathbf{\bar{a}}$ and $\mathbf{\bar{z}}$ lies in $\bar{\mathcal{S}}$. Let $\mathbf{b}$ be the \icl{ point of intersection between the line that includes the line segment $L$}, and $M_{\bar{\mathcal{S}}}(\mathbf{z}(0))$. Then the definition of $M_{\bar{\mathcal{S}}}(\mathbf{z}(0))$ tells us that the \icl{this line} 
meets $M_{\bar{\mathcal{S}}}(\mathbf{z}(0))$ at a right \icl{angles}. $d(\mathbf{b},\mathbf{\bar{z}})$ is constant and $d(\mathbf{z}(t),\mathbf{\bar{z}})$ \icl{is constant} as $\mathbf{z}(t)\in\mathcal{S}$, which implies that $d(\mathbf{z}(t),\mathbf{\bar{a}})$ is also constant (as illustrated in \autoref{fig:orthogonality}). Indeed, we have
\begin{equation}
\begin{aligned}
d(\mathbf{z}(t),\mathbf{\bar{a}})^2&=d(\mathbf{z}(t),\mathbf{b})^2+d(\mathbf{b},\mathbf{\bar{a}})^2 \\
&=d(\mathbf{z}(t),\mathbf{\bar{z}})^2-d(\mathbf{b},\mathbf{\bar{z}})^2+d(\mathbf{b},\mathbf{\bar{a}})^2
\end{aligned}
\end{equation}
and all the terms on the right hand side are constant.
\end{proof}

\begin{figure}[h]
\centering
\begin{tikzpicture}[scale=0.6,anchor=mid,>=latex',line join=bevel,]
\coordinate [label=left:\textcolor{blue}{$\mathbf{b}$}] (b) at (0,0);
\coordinate [label=below:\textcolor{blue}{$\mathbf{\bar{a}}$}] (abar) at (1,0);
\coordinate [label=below:\textcolor{blue}{$\mathbf{\bar{z}}$}] (zbar) at (3,0);
\coordinate [label=below:{$L$}] (L) at ($(abar) ! 0.5 ! (zbar)$);
\coordinate [label=left:\textcolor{blue}{$\mathbf{z}$}] (z) at (0,2);
\fill (b) circle (3pt);
\fill (abar) circle (3pt);
\fill (zbar) circle (3pt);
\fill (z) circle (3pt);

\draw [dashed] (b) -- (abar);
\draw (abar) -- (zbar);

\draw (z) -- (zbar);
\draw (z) -- (abar);
\draw [dashed] [<->] ($(b) ! 1.7 ! (z)$) -- ($ (z) ! 2.5 ! (b) $);
\coordinate [label=right:{$M_{\bar{\mathcal{S}}}(\mathbf{z})$}] (M) at ($(b) ! 1.5 ! (z)$);
\draw [right angle symbol={z}{b}{zbar}];
\end{tikzpicture}
\caption{$\mathbf{\bar{a}}$ and $\mathbf{\bar{z}}$ are saddle points of $\varphi$ which is $C^2$ and concave-convex on $\mathbb{R}^{n+m}$, and $L$ is the line segment between them. $\mathbf{z}$ is a point on a solution in $\mathcal{S}_{\mathbf{\bar{z}}}$ which lies on $M_{\mathcal{\bar{S}}}(\mathbf{z})$ which is orthogonal to $L$ by definition. $\mathbf{b}$ is the point of intersection between $M_{\bar{\mathcal{S}}}(\mathbf{z})$ and the extension of $L$.%
}\label{fig:orthogonality}
\end{figure}
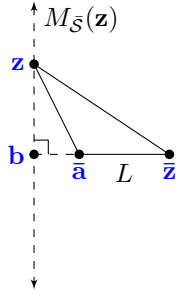
Using these orthogonality results we prove the key result of the section, a convexity result between $\mathcal{S}_{\mathbf{\bar{z}}}$ and $\mathbf{\bar{z}}$.
\begin{lemma}\label{combined-convexity-of-S-and-bar-S-fullspace}
Let $\varphi$ be $C^2$ and concave-convex on $\mathbb{R}^{n+m}$, $\mathbf{\bar{z}}\in\bar{\mathcal{S}}$ and $\mathbf{z}(t)\in\mathcal{S}_{\mathbf{\bar{z}}}$. Then for any $s\in[0,1]$, the convex combination $\mathbf{z}'(t)=(1-s)\mathbf{\bar{z}}+s\mathbf{z}(t)$ lies in $\mathcal{S}_{\mathbf{\bar{z}}}$.  If in addition $\mathbf{z}\in\mathcal{S}$, then $\mathbf{z'}(t)\in\mathcal{S}$.
\end{lemma}
\begin{proof}
Clearly $\mathbf{z'}$ is a constant distance from $\mathbf{\bar{z}}$. We must show that $\mathbf{z}'(t)$ is also a solution to \eqref{gradmethod-fullspace}. We argue in a similar way to \autoref{fig:twoplanes} but with spheres instead of planes. %
Let the solution to \eqref{gradmethod-fullspace} starting at $\mathbf{z}'(0)$ be denoted $\mathbf{z}''(t)$. We must show this is equal to $\mathbf{z}'(t)$. %
As $\mathbf{z}(t)\in\mathcal{S}$ it lies on a sphere about $\mathbf{\bar{z}}$, say of radius $r$, and by construction $\mathbf{z}'(0)$ lies on a smaller sphere about $\mathbf{\bar{z}}$ of radius $rs$. By \autoref{incrementally-stable-fullspace}, $d(\mathbf{z}(t),\mathbf{z}''(t))$ and $d(\mathbf{z}''(t),\mathbf{\bar{z}})$ are non-increasing, so that $\mathbf{z}''(t)$ must be within $rs$ of $\mathbf{\bar{z}}$ and within $r(1-s)$ of $\mathbf{z}(t)$. The only such point is $\mathbf{z}'(t)=(1-s)\mathbf{\bar{z}}+s\mathbf{z}(t)$ which proves the claim.
For the additional statement, we consider another saddle point $\mathbf{\bar{a}}\in\bar{\mathcal{S}}$ and let $L$ be the line segment connecting $\mathbf{\bar{a}}$ and $\mathbf{\bar{z}}$. By \autoref{S-orthogonal-to-barS-fullspace}, $\mathbf{z}(t)$ lies in $M_{\bar{\mathcal{S}}}(\mathbf{z}(0))$, so by construction, $\mathbf{z'}(t)\in M_{\bar{\mathcal{S}}}(\mathbf{z'}(0))$, (as illustrated by \autoref{fig:rescaled-orthogonality}). Hence, by \autoref{Orthogonal-and-constant-distance-from-one-implies-in-S}, $\mathbf{z'}(t)\in\mathcal{S}$.
\end{proof}
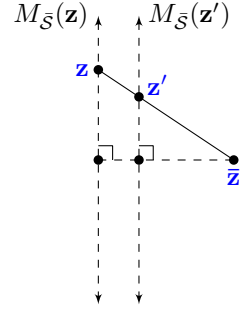
\begin{figure}[h]
\centering
\begin{tikzpicture}[scale=0.6,anchor=mid,>=latex',line join=bevel,]
\coordinate (b) at (0,0);
\coordinate [label=below:\textcolor{blue}{$\mathbf{\bar{z}}$}] (zbar) at (3,0);
\coordinate [label=left:\textcolor{blue}{$\mathbf{z}$}] (z) at (0,2);
\coordinate [label={[xshift=0.25cm, yshift=-0.1cm]\textcolor{blue}{$\mathbf{z'}$}}] (zprime) at ($(z) ! 0.3 ! (zbar)$);
\coordinate (bprime) at ($(b)!(zprime)!(zbar)$);
\coordinate [label=left:{$M_{\bar{\mathcal{S}}}(\mathbf{z})$}] (M) at ($(b) ! 1.6 ! (z)$);
\coordinate [label=right:{$M_{\bar{\mathcal{S}}}(\mathbf{z'})$}] (Mprime) at ($(bprime)!(M)!(zprime)$);
\fill (b) circle (3pt);
\fill (zbar) circle (3pt);
\fill (z) circle (3pt);
\fill (zprime) circle (3pt);
\fill (bprime) circle (3pt);

\draw [dashed] (b) -- (zbar);

\draw (z) -- (zbar);
\draw [dashed] [<->] (Mprime) -- ($(Mprime) ! 2 ! (bprime)$);
\draw [dashed] [<->] (M) -- ($(M) ! 2 ! (b)$);
\draw [right angle symbol={z}{b}{zbar}];
\draw [right angle symbol={zprime}{bprime}{zbar}];
\end{tikzpicture}
\caption{$\mathbf{\bar{z}}$ is a saddle point of $\varphi$ which is $C^2$ and concave-convex on $\mathbb{R}^{n+m}$. $\mathbf{z}$ is a point on a solution in $\mathcal{S}$ and $\mathbf{z'}$ is a convex combination of $\mathbf{z}$ and $\mathbf{\bar{z}}$. $M_{\bar{\mathcal{S}}}(\mathbf{z})$ and $M_{\bar{\mathcal{S}}}(\mathbf{z'})$ are parallel to each other by definition. 
}\label{fig:rescaled-orthogonality}
\end{figure}

\icl{In \appendixref{sec:Classification-of-S} we
also prove that the set $\mathcal{S}$ is convex (stated as \autoref{S-is-convex}).}

\section{Classification of \texorpdfstring{$\mathcal{S}$}{S}}\label{sec:Classification-of-S}
We will now proceed with a full classification of $\mathcal{S}$ and prove Theorems \ref{Classification-in-terms-of-B(z)z}-\ref{Classification-of-S-in-relaxed-linear-constraints-case}. For notational convenience we will make the assumption (without loss of generality) that $\mathbf{0}\in\bar{\mathcal{S}}$. Then we compute $\varphi_x(\mathbf{z}),\varphi_y(\mathbf{z})$ from line integrals from $\mathbf{0}$ to $\mathbf{z}$. Indeed, letting $\mathbf{\hat{z}}$ be a unit vector parallel to $\mathbf{z}$, we have
\begin{equation}\label{line-integral-form-of-time-derivative}
\!\!\!\!\!\begin{bmatrix}
   \varphi_x(\mathbf{z})\\
   -\varphi_y(\mathbf{z})
  \end{bmatrix}
=\left(\int_0^{|\mathbf{z}|}\begin{bmatrix}
\varphi_{xx}(s\mathbf{\hat{z}}) &\!\! \varphi_{xy}(s\mathbf{\hat{z}})\\
-\varphi_{yx}(s\mathbf{\hat{z}}) &\!\! -\varphi_{yy}(s\mathbf{\hat{z}})
\end{bmatrix}
ds
\right)\mathbf{\hat{z}}.
\end{equation}
Together with the definition of the matrices $\mathbf{A}(\mathbf{z})$ and $\mathbf{B}(\mathbf{z})$ given by \eqref{def-of-A-and-B} we obtain
\begin{equation}\label{eq:line-integral-form-of-derivative-without-B'}
  \begin{bmatrix}
   \varphi_x(\mathbf{z})\\
   -\varphi_y(\mathbf{z})
  \end{bmatrix}
  =\int^{|\mathbf{z}|}_0(\mathbf{A}(s\mathbf{\hat{z}})+\mathbf{B}(s\mathbf{\hat{z}}))\mathbf{\hat{z}}\,ds.
\end{equation}
We are now ready to prove the first main result.
\begin{proof}[Proof of \autoref{Classification-in-terms-of-B(z)z}]
Define the set $\mathcal{X}$ as solutions of the ODE \eqref{eq:limiting-ODE-fullspace} which obey the condition \eqref{eq:first-thm-kernel-condition} for all $t\in\mathbb{R}$ and $r\in[0,1]$. Then \autoref{Classification-in-terms-of-B(z)z} is the statement that $\mathcal{X}=\mathcal{S}$. For brevity we define the matrix $\mathbf{B}'(\mathbf{z})$ by
\begin{equation}\label{eq:B'}
\mathbf{B'}(\mathbf{z})=\mathbf{B}(\mathbf{z})+(\mathbf{A}(\mathbf{z})-\mathbf{A}(\mathbf{0})).
\end{equation}
As $\mathbf{A}(\mathbf{z})$ is skew symmetric and $\mathbf{B}(\mathbf{z})$ is symmetric we have
$\ker(\mathbf{B'}(\mathbf{z}))\hspace{-.2mm}=\hspace{-.1mm}\ker(\mathbf{B}(\mathbf{z}))\cap\ker(\mathbf{A}(\mathbf{z})\hspace{-.1mm}-\hspace{-.1mm}\mathbf{A}(\mathbf{0}))$, so that condition \eqref{eq:first-thm-kernel-condition} is equivalent to
\begin{equation}
\mathbf{z}(t)\in \ker(\mathbf{B}'(r\mathbf{z}(t)))\text{   for all }t\in\mathbb{R}, r\in[0,1].
\end{equation}

We will prove that $\mathcal{X}\subseteq\mathcal{S}_{\mathbf{0}}$, $\mathcal{X}\subseteq\mathcal{S}$ and $\mathcal{S}_{\mathbf{0}}\subseteq\mathcal{X}$. As the other inclusion $\mathcal{S}\subseteq\mathcal{S}_{\mathbf{0}}$ is clear this will prove the theorem.

\textbf{Step 1: $\mathcal{X}\subseteq\mathcal{S}_{\mathbf{0}}$.} For any non-zero point $\mathbf{z}$ we can compute the partial derivatives of $\varphi$ at $\mathbf{z}$ using the line integral formula \eqref{eq:line-integral-form-of-derivative-without-B'} and \eqref{eq:B'},
\begin{equation}\label{line-integral-form-of-derivative-with-A-and-B}
 \begin{aligned}
  \begin{bmatrix}
   \varphi_x(\mathbf{z})\\
   -\varphi_y(\mathbf{z})
  \end{bmatrix}
&=\mathbf{A}(\mathbf{0})\mathbf{z}+\int^{|\mathbf{z}|}_0\mathbf{B'}(s\mathbf{\hat{z}})\mathbf{\hat{z}}ds
\end{aligned}
\end{equation}
where $\mathbf{z}=|\mathbf{z}|\mathbf{\hat{z}}$. If $\mathbf{z}(t)\in\mathcal{X}$, then $\dot{\mathbf{z}}(t)=\mathbf{A}(\mathbf{0})\mathbf{z}(t)$, and by skew-symmetry of $\mathbf{A}(\mathbf{0})$, $|\mathbf{z}(t)|$ is constant, which means that $\mathbf{z}(t)$ is a constant distance from $\mathbf{0}$. Furthermore, the assumption that $\mathbf{z}(t)\in\ker(\mathbf{B}'(r\mathbf{z}(t)))$ for $r\in[0,1]$ implies that the integrand in \eqref{line-integral-form-of-derivative-with-A-and-B} vanishes, and $\mathbf{z}(t)$ is a solution of the gradient method.

\textbf{Step 2: $\mathcal{X}\subseteq\mathcal{S}$.} Let $\bar{\mathbf{z}}$ be arbitrary. Consider the function $t\mapsto d(\mathbf{z}(t),\mathbf{\bar{z}})^2$. By expanding in the orthonormal basis of eigenvectors of $\mathbf{A}(\mathbf{0})$ we observe that this function is a linear combination of continuous periodic functions. As, by \autoref{incrementally-stable-fullspace}, this function is also non-increasing, it must be constant. 

\textbf{Step 3: $\mathcal{S}_{\mathbf{0}}\subseteq\mathcal{X}$.} Let $\mathbf{z}(t)\in\mathcal{S}_{\mathbf{0}}$ and $R=|\mathbf{z}(t)|$ which is constant. For $r\in[0,R]$, define $\mathbf{z}(t;r)=(r/R)\mathbf{z}(t)$, so that $\mathbf{z}(t;0)=\mathbf{0}$ and $\mathbf{z}(t;R)=\mathbf{z}(t)$. Note that the corresponding unit vector $\mathbf{\hat{z}}(t;r)=\mathbf{\hat{z}}(t)$ does not depend on $r$. The convexity result \autoref{combined-convexity-of-S-and-bar-S-fullspace} implies that $\mathbf{z}(t;r)\in\mathcal{S}_{\mathbf{0}}$, and is a solution of the gradient method. We shall compute the time derivative of this in two ways. First, we use \eqref{gradmethod-fullspace} and \eqref{line-integral-form-of-derivative-with-A-and-B} to obtain,
\begin{equation}\label{eq:dotztr1}
\dot{\mathbf{z}}(t;r)=\mathbf{A}(\mathbf{0})\mathbf{z}(t;r)+\int^r_0\mathbf{B}'(s\mathbf{\hat{z}}(t))\mathbf{\hat{z}}(t)\,ds.
\end{equation}
Second, we use the explicit definition of $\mathbf{z}(t;r)$ in terms of $\mathbf{z}(t)$ to obtain,
\begin{equation}\label{eq:dotztr2}
\dot{\mathbf{z}}(t;r)=\frac rR\mathbf{A}(\mathbf{0})\mathbf{z}(t)+\frac rR\int^R_0\mathbf{B}'(s\mathbf{\hat{z}}(t))\mathbf{\hat{z}}(t)\,ds.
\end{equation}
Equating \eqref{eq:dotztr1} and \eqref{eq:dotztr2} we deduce that
\begin{equation}
\int^r_0\mathbf{B}'(s\mathbf{\hat{z}}(t))\mathbf{\hat{z}}(t)\,ds=\frac rR\int^R_0\mathbf{B}'(s\mathbf{\hat{z}}(t))\mathbf{\hat{z}}(t)\,ds.
\end{equation}
Differentiating with respect to $r$ we have,
\begin{equation}
\mathbf{B}'(r\mathbf{\hat{z}}(t))\mathbf{\hat{z}}(t)=\frac1R\int^R_0\mathbf{B}'(s\mathbf{\hat{z}}(t))\mathbf{\hat{z}}(t)\,ds.
\end{equation}
The right hand side of this is independent of $r$, which implies that the left hand side is also independent of $r$, and is thus equal to its value at $r=0$, so that
\begin{equation}\label{eq:B'rzhat=B0zhat}
\mathbf{B}'(r\mathbf{\hat{z}}(t))\mathbf{\hat{z}}(t)=\mathbf{B}'(\mathbf{0})\mathbf{\hat{z}}(t)=\mathbf{B}(\mathbf{0})\mathbf{\hat{z}}(t).
\end{equation}
Putting this back into our expression for $\dot{\mathbf{z}}$ we find that
\begin{equation}
\dot{\mathbf{z}}(t)=\mathbf{A}(\mathbf{0})\mathbf{z}(t)+\mathbf{B}(\mathbf{0})\mathbf{z}(t),
\end{equation}
but as $|\mathbf{z}(t)|$ is constant, $\mathbf{A}(\mathbf{0})$ skew symmetric, and $\mathbf{B}(\mathbf{0})$ symmetric, $\mathbf{B}(\mathbf{0})\mathbf{z}(t)$ must vanish, which, together with \eqref{eq:B'rzhat=B0zhat} shows that $\mathbf{z}(t)\in\mathcal{X}$.
\end{proof}
\icl{
}
\icl{The following Lemma follows directly from \autoref{Classification-in-terms-of-B(z)z} and will be used to prove \autoref{S-is-convex}, which states that the set $\mathcal{S}$ is convex.}
\begin{lemma}\label{solutions-in-S-are-a-constant-distance-apart}
Let $\varphi$ be $C^2$ and concave-convex on $\mathbb{R}^{n+m}$. Let $\mathbf{z}(t),\mathbf{z}'(t)\in\mathcal{S}$. Then $d(\mathbf{z}(t),\mathbf{z}'(t))$ is constant.
\end{lemma}
\begin{proof}
Using \autoref{Classification-in-terms-of-B(z)z} we have that $\mathbf{z}(t)-\mathbf{z}'(t)=e^{t\mathbf{A}(\mathbf{0})}(\mathbf{z}(0)-\mathbf{z}'(0))$ which has constant magnitude as $\mathbf{A}(\mathbf{0})$ is skew symmetric.
\end{proof}
\begin{proposition}\label{S-is-convex}
Let $\varphi$ be $C^2$ and concave-convex on $\mathbb{R}^{n+m}$, then $\mathcal{S}$ is convex.
\end{proposition}
\begin{proof}
The proof is very similar to that of \autoref{combined-convexity-of-S-and-bar-S-fullspace}.
Let $\mathbf{z}(t),\mathbf{z}'(t)\in\mathcal{S}$, and $s\in(0,1)$. Set $\mathbf{w}(t)=s\mathbf{z}(t)+(1-s)\mathbf{z}'(t)$. By \autoref{solutions-in-S-are-a-constant-distance-apart} we know that $d=d(\mathbf{z}(t),\mathbf{z}'(t))$ is constant. Denote the solution of the gradient method starting from $\mathbf{w}(0)$ as $\mathbf{w}'(t)$. We must prove that $\mathbf{w}'(t)=\mathbf{w}(t)$ and that $\mathbf{w}(t)\in\mathcal{S}$. First we imagine two closed balls centered on $\mathbf{z}(t)$ and $\mathbf{z}'(t)$ and of radii $sd$ and $(1-s)d$ respectively. By \autoref{incrementally-stable-fullspace}, $\mathbf{w}'(t)$ is constrained to lie within both of these balls. For each $t$ there is only one such point and it is exactly $\mathbf{w}(t)$. Next we let $\mathbf{\bar{a}}\in\bar{\mathcal{S}}$ be arbitrary, then $d(\mathbf{\bar{a}},\mathbf{w}(t))$ is determined by $d(\mathbf{z}(t),\mathbf{z'}(t)),d(\mathbf{\bar{a}},\mathbf{z})$ and $d(\mathbf{\bar{a}},\mathbf{z'}(t))$, (as illustrated by \autoref{fig:uniquely-determined-triangle}). Indeed, we may assume by translation that $\mathbf{\bar{a}}=\mathbf{0}$, and then
\begin{equation}\label{eq:s-convex-algebra}
\begin{aligned}
d&(\mathbf{\bar{a}},\mathbf{w}(t))^2=d(\mathbf{0},\mathbf{z}(t)+(1-s)\mathbf{z}'(t))^2\\
&=s^2d(\mathbf{0},\mathbf{z}(t))^2+\!(1\!-\!s)^2d(\mathbf{0},\mathbf{z'}(t))^2-\!2s(1\!-\!s)\mathbf{z}^T(t)\mathbf{z}'(t)
\end{aligned}
\end{equation}
The first two terms in \eqref{eq:s-convex-algebra} are constant by \autoref{solutions-in-S-are-a-constant-distance-apart} and the third can be computed as
\begin{equation}
2\mathbf{z}^T(t)\mathbf{z}'(t)=d(\mathbf{z}(t),\mathbf{z'}(t))^2-d(\mathbf{0},\mathbf{z}(t))^2-d(\mathbf{0},\mathbf{z'}(t))^2
\end{equation}
which is constant for the same reason.
\end{proof}

\begin{figure}[h]
\centering
\begin{tikzpicture}[scale=0.6,anchor=mid,>=latex',line join=bevel,]
\coordinate [label=left:\textcolor{blue}{$\mathbf{\bar{a}}$}] (abar) at (2,0);
\coordinate [label=left:\textcolor{blue}{$\mathbf{z}$}] (z) at (0,4);
\coordinate [label=right:\textcolor{blue}{$\mathbf{z'}$}] (zprime) at (5,3);
\coordinate [label=above:\textcolor{blue}{$\mathbf{w}$}] (w) at ($(z)!0.3!(zprime)$);
\fill (abar) circle (3pt);
\fill (z) circle (3pt);
\fill (zprime) circle (3pt);
\fill (w) circle (3pt);

\draw [<->] (z) -- (w);
\draw [<->] (w) -- (zprime);
 \draw [<->] (zprime) -- (abar);
\draw [<->] (abar) -- (z);
\draw [<->] (w) -- (abar);
\end{tikzpicture}
\caption{$\mathbf{z}$ and $\mathbf{z}'$ are two elements of $\mathcal{S}$ and $\mathbf{w}$ is a convex combination of them. $\mathbf{\bar{a}}$ is a saddle point in $\bar{\mathcal{S}}$. We know all the distances are constant except possibly $d(\mathbf{w},\mathbf{\bar{a}})$, but this is uniquely determined by the other four distances.}\label{fig:uniquely-determined-triangle}
\end{figure}
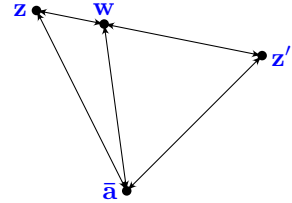


\begin{proof}[\icl{Proof of \autoref{cor:local}}]
\icl{The only if part is trivial. To prove the if part, assume there exists a trajectory in $\mathcal{S}$ that is not a saddle point, i.e. it is at a constant non-zero distance from each saddle point. Since saddle points are in $\mathcal{S}$, the convexity of $\mathcal{S}$ (\autoref{S-is-convex}) implies the existence of trajectories in $\mathcal{S}$ that are not saddle points and can be chosen to be arbitrarily close to the saddle point $\mathbf{\bar{z}}$. This contradicts the assumption in \autoref{cor:local}.}
\end{proof}
To prove \autoref{prop:unstable-to-noise} we require the following lemma which shows the existence of a conserved quantity of the gradient dynamics.
\begin{lemma}\label{lem:conserved-quantity-for-noise-proof}
Let $\varphi$ be $C^2$ and concave-convex on $\mathbb{R}^{n+m}$. Suppose that $\mathcal{S}$ contains a bi-infinite line $L=\{\mathbf{a}+s\mathbf{v}:s\in\mathbb{R}\}$. Assume that $\mathbf{0}\in\bar{\mathcal{S}}$. Then $W(t;\mathbf{z})=|(e^{t\mathbf{A}(\mathbf{0})}\mathbf{v})^T\mathbf{z}|^2$ is a conserved quantity for any solution $\mathbf{z}$ of \eqref{gradmethod-fullspace}.
\end{lemma}
\begin{proof}
As $\mathcal{S}$ is closed and convex (\autoref{S-is-convex}) we may assume that the line passes though the origin and take $\mathbf{a}=\mathbf{0}$. Let $\mathbf{v}(t)=e^{t\mathbf{A}(\mathbf{0})}\mathbf{v}$ and note that $\lambda \mathbf{v}(t)$ is a solution to the gradient method \eqref{gradmethod-fullspace} by \autoref{Classification-in-terms-of-B(z)z} for any $\lambda\in\mathbb{R}$. We follow the strategy of the first part of the proof of \autoref{infinite-line-of-saddles} with $-\lambda\mathbf{v}(t),\lambda\mathbf{v}(t)$ replacing the saddle points $\mathbf{\bar{a}}$,$\mathbf{\bar{b}}$. Indeed, let $\mathbf{z}(t)$ be any solution to \eqref{gradmethod-fullspace} and let $\lambda'=\mathbf{v}^T\mathbf{z}(0)$. Then for any $t\ge0$, \autoref{incrementally-stable-fullspace} implies that $\mathbf{z}(t)$ must satisfy
\begin{equation}\label{eq:conserved-quantity-proof}
d(\pm\lambda\mathbf{v}(t),\mathbf{z}(t))\le d(\pm\lambda\mathbf{v}(0),\mathbf{z}(0)),
\end{equation}
where by $\pm$ we mean that the equation holds for each of $+$ and $-$. In the same way as in the proof of \autoref{infinite-line-of-saddles}, taking the intersection of these balls for a sequence $\lambda\to\infty$ we deduce that $\mathbf{z}(t)$ is contained in the linear manifold normal to the line through the origin and $\mathbf{v}(t)$, and passing through $\lambda'\mathbf{v}(t)$. Indeed, by squaring \eqref{eq:conserved-quantity-proof} and expanding we obtain
\begin{equation*}
|\mathbf{z}(t)|^2\mp2\lambda \mathbf{v}(t)^T\mathbf{z}(t)\le |\mathbf{z}(0)|^2\mp 2\lambda \mathbf{v}(0)^T\mathbf{z}(0).
\end{equation*}
By dividing through by $\lambda$ and taking the limit $\lambda\to\infty$ we deduce that $\mathbf{v}(t)^T\mathbf{z}(t)$ is equal to $\mathbf{v}(0)^T\mathbf{z}(0)$ which implies that $W(t;\mathbf{z})$ is conserved.
\end{proof}
\begin{proof}[Proof of \autoref{prop:unstable-to-noise}]
Consider the conserved quantity $W(t;\mathbf{z})$ given by \autoref{lem:conserved-quantity-for-noise-proof}. Applying It\=o's lemma and taking expectations, we have
 \begin{equation*}
\begin{aligned}
 \frac{d}{dt}\mathbb{E}W(t;\mathbf{z}(t))&=\mathbb{E}\dot{W}(t;\mathbf{z}(t))+\tfrac12\mathbb{E}\operatorname{Tr}(\mathbf{\Sigma}^TW_{\mathbf{z}\mathbf{z}}\mathbf{\Sigma})
 \end{aligned}
 \end{equation*}
 where $\mathbf{\Sigma}=\diag(\Sigma^x,\Sigma^y)$, $\dot{W}$ is the total derivative along the deterministic flow \eqref{gradmethod-fullspace} and $\operatorname{Tr}$ is the trace operator. As $W$ is conserved along the deterministic flow, $\dot{W}=0$ and a simple computation shows that the second term is independent of $\mathbf{z}$ and bounded below by a strictly positive constant. Therefore $\mathbb{E}W(t;\mathbf{z}(t))$ grows at least linearly in time. It remains to note that $W(t;\mathbf{z})\le |e^{t\mathbf{A}(\mathbf{0})}\mathbf{v}|^2|\mathbf{z}|^2\le |\mathbf{v}|^2|\mathbf{z}|^2$, {\color{black}so that $|\mathbf{z}(t)|^2\ge cW(t;\mathbf{z}(t))$} for a constant $c>0$. This implies that also $\mathbb{E}|\mathbf{z}(t)|^2\to\infty$ and completes the proof of the proposition.
\end{proof}

\icl{To prove} \autoref{S-inside-Slinear} and \autoref{Classification-of-S-in-relaxed-linear-constraints-case} we {\color{black}make use of the following result that follows easily from linear algebra arguments.}
\begin{lemma}\label{largest-invariant-subspace}
Let $X$ be a linear subspace of $\mathbb{R}^n$ and $A\in\mathbb{R}^{n\times n}$ a normal matrix. Let
\begin{equation}
Y=\Span\{v\in X: v \text{ is an eigenvector of }A\}.
\end{equation}
Then $Y$ is the largest subset of $X$ that is invariant under $A$.
\end{lemma}
We note that invariance of a subspace under $A$ is equivalent to invariance of the subspace under the group $e^{tA}$.
\begin{proof}[Proof of \autoref{S-inside-Slinear}]
\textbf{Step 1: $\mathcal{S}_{\text{linear}}\subseteq\mathcal{S}$ when $\varphi$ is a quadratic function.} We will use the characterisation of $\mathcal{S}$ given by \autoref{Classification-in-terms-of-B(z)z}. By \autoref{largest-invariant-subspace}, $\mathcal{S}_{\text{linear}}$ is invariant under $e^{t\mathbf{A}(\mathbf{0})}$, so that $\mathbf{z}(0)\in\mathcal{S}_{\text{linear}}\implies \mathbf{z}(t)=e^{t\mathbf{A}(\mathbf{0})}\mathbf{z}(0)\in\mathcal{S}_{\text{linear}}$. Hence if $\mathbf{z}(0)\in\mathcal{S}_{\text{linear}}$ then $\mathbf{z}(t)\in\ker(\mathbf{B'}(\mathbf{0}))$ for all time $t$, and as $\varphi$ is a quadratic function, $\mathbf{B'}(\mathbf{z})$ is constant, so this is enough to show $\mathcal{S}_{\text{linear}}\subseteq\mathcal{S}$.

\textbf{Step 2: $\mathcal{S}\subseteq\mathcal{S}_{\text{linear}}$.} Let $\mathbf{z}(t)\in\mathcal{S}$, then by \autoref{Classification-in-terms-of-B(z)z} taking $r=0$ we have $\mathbf{z}(t)=e^{t\mathbf{A}(\mathbf{0})}\in\ker(\mathbf{B'}(\mathbf{0}))$ for all $t\in\mathbb{R}$. Thus $\mathcal{S}$ lies inside the largest subset of $\ker(\mathbf{B'}(\mathbf{0}))$ that is invariant under the action of the group $e^{t\mathbf{A}(\mathbf{0})}$, which by \autoref{largest-invariant-subspace} is exactly $\mathcal{S}_{\text{linear}}$.
\end{proof}
In order to prove \autoref{Classification-of-S-in-relaxed-linear-constraints-case} we give a different interpretation of the condition in \autoref{Classification-in-terms-of-B(z)z}. The condition $\mathbf{z}\in\ker(\mathbf{B}(s\mathbf{z}))$ for all $s\in[0,1]$ looks like a line integral condition. Indeed, if we define a function $V(\mathbf{z})$ by%
\begin{equation}
V(\mathbf{z})=\mathbf{z}^T\left(\int^1_0\int^1_0\mathbf{B}(ss'\mathbf{z})s\,ds'\,ds\right)\mathbf{z}
\end{equation}
then as $\mathbf{B}(\mathbf{z})$ is symmetric negative semi-definite we have that $V(\mathbf{z})=0$ if and only if $\mathbf{z}\in\ker(\mathbf{B}(s\mathbf{z}))$ for every $s\in[0,1]$. This still leaves the condition $\mathbf{z}\in\ker(\mathbf{A}(s\mathbf{z})-\mathbf{A}(\mathbf{0}))$ for all $s\in[0,1]$, and the function $V$ has no natural interpretation in general. However in the specific case where $\varphi$ is the Lagrangian of a concave optimization problem where the relaxed constraints are linear, we do have an interpretation.
In this case the assumption that $\mathbf{0}$ is a saddle point is no longer generic and we must translate coordinates explicitly. Let the Lagrangian of the optimization problem be given by
\begin{equation}\label{lagconds}
\begin{aligned}
\varphi(x',y')&=U'(x')+{y'}^Tg'(x') \\
U'\in C^2\text{ and concave,}& \quad g' \text{ linear with }g_x'=D.
\end{aligned}
\end{equation}
We pick a saddle point $(\bar{x}',\bar{y}')$, and shift to new coordinates $(x,y)=(x'-\bar{x}',y'-\bar{y}')$ so that $(0,0)$ is a saddle point in the new coordinates. After expanding we obtain
\begin{equation}
\varphi(x,y)=(U'(x+\bar{x}')+\bar{y}^{\prime T}g'(x+\bar{x}'))+y^Tg'(x+\bar{x}')
\end{equation}
which is a Lagrangian originating from the utility function
\begin{equation}\label{eq:U-after-change-of-variables}
U(x)=U'(x+\bar{x}')+\bar{y}^{\prime T}g'(x+\bar{x}')
\end{equation}
and constraints $g(x)=g'(x+\bar{x}')$. Without loss of generality we assume that $U(0)=0$. As $g(x)$ is a linear function we have
\begin{equation}
\mathbf{B}(\mathbf{z})=\begin{bmatrix}
                        U_{xx}(x)&0\\
			0&0
                       \end{bmatrix}
\end{equation}
so that $V(\mathbf{z})$ is independent of $y$, and in fact by direct computation we have $V(\mathbf{z})=U(x)$. This leads us to the following lemma.
\begin{lemma}\label{S-in-terms-of-U^-1(0)}
Let \eqref{lagconds} hold. Then $\mathcal{S}$ is the largest subset of $U^{-1}(\{0\})\times\mathbb{R}^m=\{(x,y)\in\mathbb{R}^{n+m}:U(x)=0\}$ that is invariant under evolution by the group $e^{t\mathbf{A}(\mathbf{0})}$, where $U$ is given by \eqref{eq:U-after-change-of-variables}.
\end{lemma}
\begin{proof}
Denote the set defined in the lemma as $\mathcal{Y}$.

\textbf{Step 1: $\mathcal{S}\subseteq\mathcal{Y}$.}
By the computation above we know that $\mathbf{z}\in U^{-1}(\{0\})\times{\color{black}\mathbb{R}^{m}}$ if and only if $\mathbf{z}\in\ker(\mathbf{B}(s\mathbf{z}))$ for all $s\in[0,1]$. Thus by \autoref{Classification-in-terms-of-B(z)z}, we have $\mathcal{S}\subseteq U^{-1}(\{0\})\times\mathbb{R}^m$ {\color{black}as $\mathcal{S}$ is} invariant under the action of $e^{t\mathbf{A}(\mathbf{0})}$.

\textbf{Step 2: $\mathcal{Y}\subseteq\mathcal{S}$.} If $\mathbf{z}(0)$ is in the largest subset of $U^{-1}(\{0\})\times\mathbb{R}^m$ invariant under the action of $e^{t\mathbf{A}(\mathbf{0})}$, then $\mathbf{z}(t)$ is in this set for all $t\in\mathbb{R}$. Defining $\mathbf{z}(t)=e^{t\mathbf{A}(\mathbf{0})}\mathbf{z}(0)$, we have $\mathbf{z}(t)\in\ker(\mathbf{B}(s\mathbf{z}(t)))$ for all $s\in[0,1]$, so $\mathbf{z}(t)\in\mathcal{S}$ by \autoref{Classification-in-terms-of-B(z)z}.
\end{proof}
To obtain a more exact expression for $\mathcal{S}$, we make use of the assumption that $U$ is analytic.
\begin{lemma}\label{Classification-of-U^-1(0)-in-analytic-case}
Let \eqref{lagconds} hold and in addition $U$ given by \eqref{eq:U-after-change-of-variables} be analytic. Then
\begin{enumerate}[(i)]
 \item $U^{-1}(\{0\})=\Span(U^{-1}(\{0\}))$.
 \item $\mathcal{S}=\{e^{t\mathbf{A}(\mathbf{0})}\mathbf{z}(0):\mathbf{z}(0)\in\mathcal{Q}\}$ where
\begin{equation}
\begin{aligned}
&\mathcal{Q}=\Span\{ (x,y)\in U^{-1}(\{0\})\times\mathbb{R}^m: \\
&(x,y)\text{ is an eigenvector of }\begin{bmatrix}
0&D^T\\
-D&0                                                                              \end{bmatrix}\bigg\}
\end{aligned}
\end{equation}
\end{enumerate}
\end{lemma}
\begin{proof}
We begin with (i). Recall we have assumed without loss of generality that $U(0)=0$. As $U^{-1}(\{0\})$ is the set of maxima of a concave function, it is convex. If $U^{-1}(\{0\})$ is the single point $0$, then (i) is trivial. Otherwise let $L$ be a line segment (of strictly positive length) in $U^{-1}(\{0\})$, and let $\hat{L}$ be the bi-infinite extension of $L$. Let $f$ be a linear bijection from $\mathbb{R}$ to $\hat{L}$, and let $I\subset \mathbb{R}$ be the interval in $\mathbb{R}$ given by $f^{-1}(L)$. Then $U(f(t)):\mathbb{R}\to\mathbb{R}$ is an analytic function whose restriction to $I$ vanishes. Hence $U(f(t))$ vanishes everywhere on $\mathbb{R}$, which is equivalent to $U$ vanishing on $\hat{L}$. By varying the choice of $L$, we deduce that $U^{-1}(\{0\})$ contains infinite lines in every direction in $\Span(U^{-1}(\{0\}))$ and by convexity is equal to $\Span(U^{-1}(\{0\}))$.

(ii) is a consequence of \autoref{S-in-terms-of-U^-1(0)} and \autoref{largest-invariant-subspace}.
\end{proof}
Lastly, we translate back into the original coordinates.

\begin{lemma}\label{U^-1-in-terms-of-U'}
Let \eqref{lagconds} hold and $U'$ be analytic, then
\begin{equation*}
\!U^{-1}(\{0\})=\{x\in\mathbb{R}^{n}\!:\mathbb{R}\ni s\mapsto U'(sx+\bar{x}')\text{ is linear}\}
\end{equation*}
where $U$ is given by \eqref{eq:U-after-change-of-variables}.
\end{lemma}
\begin{proof}
 Suppose that $x\in U^{-1}(\{0\})$ then by \autoref{Classification-of-U^-1(0)-in-analytic-case} $U(sx)=0$ for all $s\in\mathbb{R}$. Recall that $U-U'$ is a linear function. Hence $U'(sx+\bar{x}')$ is linear as a function of $s\in\mathbb{R}$. Now suppose that $U'(sx+\bar{x}')$ is linear as a function of $s\in\mathbb{R}$ for some $x\in\mathbb{R}^n$, then $U(sx)$ is also linear. But $U(0)=0$ and $U_x(0)=0$, as $\mathbf{0}$ is a saddle point of $\varphi$, so by linearity $U(sx)=0$ for all $s\in\mathbb{R}$.
\end{proof}
\begin{proof}[Proof of \autoref{Classification-of-S-in-relaxed-linear-constraints-case}]
 This is just a simple combination of \autoref{U^-1-in-terms-of-U'} and \autoref{Classification-of-U^-1(0)-in-analytic-case}.
\end{proof}

We now consider the case of the projected gradient method.
\begin{proof}[Proof of \autoref{thm:projected-gradient-method-result}]
We show how to adapt the proof of the results on the gradient method. We denote the set of equilibrium points of the projected gradient method as $\bar{\mathcal{S}}^{\mathbf{\Pi}}$ and similarly $\mathcal{S}^{\mathbf{\Pi}},\mathcal{S}^{\mathbf{\Pi}}_{\mathbf{\bar{z}}}$, in analogy with $\mathcal{S},\mathcal{S}_{\mathbf{\bar{z}}}$.

We first note that the projected gradient method is pathwise stable
(Lemma 22 in part II).
Together with the assumption that $\mathbf{0}\in\bar{\mathcal{S}}^{\mathbf{\Pi}}$, this means that the reasoning in \appendixref{sec:Geometry-of-S-bar-and-S} applies, and in particular a version of \autoref{combined-convexity-of-S-and-bar-S-fullspace} holds, i.e.
\begin{lemma}
Let $\varphi$ be $C^2$ and concave-convex on $\mathbb{R}^{n+m}$, $\mathbf{\Pi}\in\mathbb{R}^{(n+m)^2}$ be an orthogonal projection matrix, $\mathbf{\bar{z}}\in\bar{\mathcal{S}}^{\mathbf{\Pi}}$ and $\mathbf{z}(t)\in\mathcal{S}_{\mathbf{\bar{z}}}^{\mathbf{\Pi}}$. Then for any $s\in[0,1]$, the convex combination $\mathbf{z}'(t)=(1-s)\mathbf{\bar{z}}+s\mathbf{z}(t)$ lies in $\mathcal{S}_{\mathbf{\bar{z}}}^{\mathbf{\Pi}}$.  If in addition $\mathbf{z}\in\mathcal{S}^{\mathbf{\Pi}}$, then $\mathbf{z'}(t)\in\mathcal{S}^{\mathbf{\Pi}}$.
\end{lemma}

Equation \eqref{line-integral-form-of-time-derivative} becomes
\begin{equation*}
\mathbf{\Pi}
\begin{bmatrix}
   \varphi_x(\mathbf{z})\\
   -\varphi_y(\mathbf{z})
  \end{bmatrix}
  \mathbf{\Pi}
=\left(\int_0^{|\mathbf{z}|}\mathbf{\Pi}\begin{bmatrix}
\varphi_{xx}(s\mathbf{\hat{z}}) &\!\! \varphi_{xy}(s\mathbf{\hat{z}})\\
-\varphi_{yx}(s\mathbf{\hat{z}}) &\!\! -\varphi_{yy}(s\mathbf{\hat{z}})
\end{bmatrix}\mathbf{\Pi}
ds
\right)\mathbf{\hat{z}}
\end{equation*}
and we replace \eqref{def-of-A-and-B} with
\begin{align*}
  \!\widetilde{\mathbf{A}}(\mathbf{z})&=\mathbf{\Pi}\begin{bmatrix}
0&\!\!\!\!\!\! \varphi_{xy}(\mathbf{z})\notag\\
-\varphi_{yx}(\mathbf{z}) & 0
\end{bmatrix}\mathbf{\Pi}&\\
\widetilde{\mathbf{B}}(\mathbf{z})&=\mathbf{\Pi}\begin{bmatrix}
\varphi_{xx}(\mathbf{z})\!\!\!\!\!\! & 0\\
0 & -\varphi_{yy}(\mathbf{z})
\end{bmatrix}\mathbf{\Pi}
\end{align*}
The remainder of the {\color{black}proof carries} through unaltered {\color{black}in analogy with that of Theorem~\ref{Classification-in-terms-of-B(z)z}}.
\end{proof}
\section{The addition of constant gains}\label{sec:the-addition-of-constant-gains}
It is common in applications to consider the gradient method with constant gains, i.e.
\begin{equation}
\label{gradmethod-gains-fullspace}
\begin{aligned}
\dot{x}_i&=\gamma^x_i\varphi_{x_i}&&\text{for }i=1,\dotsc,n,\\
\dot{y}_j&=-\gamma^y_j\varphi_{y_j}&&\text{for }j=1,\dotsc,m.
\end{aligned}
\end{equation}
for $\varphi\in C^2$ a concave-convex function on $\mathbb{R}^{n+m}$ and $\gamma^x_i,\gamma^y_j$ positive constants. However, in the setting of an arbitrary concave-convex {\color{black}function}, this is not a generalisation, and it is sufficient to study the gradient method \eqref{gradmethod-fullspace} without gains, by a coordinate transformation that we now describe.

Let $\Lambda$ be a diagonal matrix defined from the gains by
\begin{equation}\label{eq:lambda-gains}
	 \mathbf{\Lambda}=\diag(\sqrt{\gamma^x_1},\dotsc,\sqrt{\gamma^x_n},\sqrt{\gamma^y_1},\dotsc,\sqrt{\gamma^y_m}).
\end{equation}
Given a concave-convex function $\varphi$ we define a new concave-convex function $\varphi'$ by
\begin{equation}
\varphi'(\mathbf{z}')=\varphi(\mathbf{\Lambda}\mathbf{z}').
\end{equation}
Let $\mathbf{z}'(t)$ be a solution to the gradient method \eqref{gradmethod-fullspace} without gains applied to $\varphi'$, then
{\color{black}
\begin{equation}
\mathbf{z}(t):=\mathbf{\Lambda}\mathbf{z}'(t)
\label{eq:coord_transf}
\end{equation}
}
is a solution to the gradient method \eqref{gradmethod-gains-fullspace} applied to $\varphi$ with gains. Indeed, we have
\begin{equation*}
\dot{\mathbf{z}}(t)=\mathbf{\Lambda} \dot{\mathbf{z}}'(t)=\mathbf{\Lambda}^2\begin{bmatrix}
                        \varphi_x(\mathbf{\Lambda} \mathbf{z}'(t))\\
			-\varphi_y(\mathbf{\Lambda} \mathbf{z}'(t))
                       \end{bmatrix}=\mathbf{\Lambda}^2\begin{bmatrix}
                                               \varphi_x(\mathbf{z}(t))\\
                       			-\varphi_y(\mathbf{z}(t))
                                              \end{bmatrix}
\end{equation*}
and the $\mathbf{\Lambda}^2$ term gives the gains.

Thus any properties of the gradient method with gains can be obtained from the gradient method without gains applied to a suitably modified function.

However, applying this transformation to the subgradient method has the effect of altering the metric in the convex projection. {\color{black}{\color{black}We therefore use} the following definition of subgradient dynamics with gains.}
\begin{definition}[Subgradient method with gains]
Given a non-empty closed convex set $K\subseteq\mathbb{R}^{n+m}$, $\varphi\in C^2$ a concave-convex function on $K$ and a set of positive gains $\gamma^x_i,\gamma^y_j$ as in \eqref{gradmethod-gains-fullspace}, we define the \textit{subgradient method on $K$ with gains} as a semi-flow on $(K,d)$ consisting of Carath\'eodory solutions of
\begin{equation}\label{gradmethod-gains-convex-domain}
\begin{aligned}
\dot{\mathbf{z}}&=\mathbf{f}(\mathbf{z})-\mathbf{P}_{N_K(\mathbf{z}),d_{\mathbf{\Lambda}^{-1}}}(\mathbf{f}(\mathbf{z}))\\
\end{aligned}
\end{equation}
where $\mathbf{f}(\mathbf{z})$ is the vector field of the gradient method with gains \eqref{gradmethod-gains-fullspace} and $\mathbf{P}_{M,d_{\mathbf{\Lambda}^{-1}}}$ is a weighted convex projection given by
\begin{equation}
\begin{aligned}
\mathbf{P}_{M,d_{\mathbf{\Lambda}^{-1}}}(\mathbf{z})&=\operatorname{argmin}_{\mathbf{w}\in M}d(\mathbf{\Lambda}^{-1}\mathbf{w},\mathbf{\Lambda}^{-1}\mathbf{z})
\end{aligned}
\end{equation}
where $\Lambda$ is defined in terms of the gains by \eqref{eq:lambda-gains}.
\end{definition}
{\color{black}It should be noted that the weighted metric used in the projection 
arises from the stretching of the domain $K$ when the coordinate transformation \eqref{eq:coord_transf} is applied.}
\begin{remark}
{\color{black}When non-negativity constraints are present the subgradient dynamics are}
not affected by this change to the metric in the convex projection, {\color{black}i.e. the dynamics in \eqref{gradmethod-gains-convex-domain} are identical to the ones where an unweighted metric is used in the projection}. For example, if the $y$ coordinates are restricted to be non-negative and the $x$ coordinates unconstrained, then the subgradient method with gains \eqref{gradmethod-gains-convex-domain} is given by
\begin{equation}
\begin{aligned}
\dot{x}_i&=\gamma^x_i\varphi_{x_i}&&\text{for }i=1,\dotsc,n,\\
\dot{y}_j&=[-\gamma^y_j\varphi_{y_j}]_{y_j}^+&&\text{for }j=1,\dotsc,m.
\end{aligned}
\end{equation}
This holds more generally for any convex set $K$ with boundaries aligned to the coordinate axes.
\end{remark}

\end{document}